\documentclass[12pt]{amsart}

\usepackage{amsthm}
\usepackage{amsmath}
\usepackage{amssymb}
\usepackage{enumerate}
\usepackage{graphicx}
\usepackage[hidelinks,pdftex]{hyperref}
\usepackage{booktabs}
\usepackage{color}
\usepackage[dvipsnames]{xcolor}
\usepackage{import}
\usepackage{tikz-cd}
\usepackage{microtype}

\usepackage[margin=3cm]{geometry}

\AtBeginDocument{%
   \def\MR#1{}
}

\numberwithin{equation}{section}
\numberwithin{figure}{section}

\theoremstyle{plain}
\newtheorem{theorem}[equation]{Theorem}
\newtheorem{corollary}[equation]{Corollary}
\newtheorem{lemma}[equation]{Lemma}

\newtheorem{proposition}[equation]{Proposition}

\newtheorem*{namedtheorem}{\theoremname}
\newcommand{\theoremname}{testing}

\theoremstyle{definition}
\newtheorem{definition}[equation]{Definition}
\newtheorem{remark}[equation]{Remark}
\newtheorem{construction}[equation]{Construction}

\newcommand{\from}{\colon} 

\newcommand{\HH}{{\mathbb{H}}}
\newcommand{\RR}{{\mathbb{R}}}
\newcommand{\ZZ}{{\mathbb{Z}}}
\newcommand{\NN}{{\mathbb{N}}}
\newcommand{\CC}{{\mathbb{C}}}

\newcommand{\calD}{{\mathcal{D}}}
\newcommand{\calF}{{\mathcal{F}}}

\newcommand{\calM}{{\mathcal{M}}}

\newcommand{\calT}{{\mathcal{T}}}

\newcommand{\refthm}[1]{Theorem~\ref{Thm:#1}}
\newcommand{\reflem}[1]{Lemma~\ref{Lem:#1}}
\newcommand{\refprop}[1]{Proposition~\ref{Prop:#1}}
\newcommand{\refcor}[1]{Corollary~\ref{Cor:#1}}

\newcommand{\refdef}[1]{Definition~\ref{Def:#1}}
\newcommand{\refsec}[1]{Section~\ref{Sec:#1}}
\newcommand{\reffig}[1]{Figure~\ref{Fig:#1}}

\newcommand{\bdy}{\partial}

\newcommand{\PSL}{\operatorname{PSL}}

\newcommand{\area}{\operatorname{area}}
\newcommand{\len}{\operatorname{len}}
\newcommand{\interior}{{\mathrm{int}}}

\title{Constructing knots with specified geometric limits}

\author{Urs Fuchs}
\author{Jessica S.~Purcell}
\author{John Stewart}
\address{School of Mathematics,
Monash University,
VIC 3800, Australia}

\begin{document}

\begin{abstract}
It is known that any tame hyperbolic 3-manifold with infinite volume and a single end is the geometric limit of a sequence of finite volume hyperbolic knot complements. Purcell and Souto showed that if the original manifold embeds in the 3-sphere, then such knots can be taken to lie in the 3-sphere.
However, their proof was nonconstructive; no examples were produced. In this paper, we give a constructive proof in the geometrically finite case. That is, given a geometrically finite, tame hyperbolic 3-manifold with one end, we build an explicit family of knots whose complements converge to it geometrically. Our knots lie in the (topological) double of the original manifold. The construction generalises the class of fully augmented links to a Kleinian groups setting. 
\end{abstract}

\maketitle

\section{Introduction}\label{Sec:Intro}

In this paper, we construct finite volume hyperbolic 3-manifolds that converge geometrically to infinite volume ones. In 2010, Purcell and Souto proved that every tame infinite volume hyperbolic 3-manifold with a single end that embeds in $S^3$ is the geometric limit of complements of knots in $S^3$~\cite{PurcellSouto}. However, that was purely an existence result; the proof shed very little light on what the knots might look like. This paper is much more constructive. Starting with a tame, infinite volume hyperbolic 3-manifold $M$ with a single end, we give an algorithm to construct a sequence of knots that converge geometrically to $M$ --- with a cost. We can no longer ensure that our knot complements lie in $S^3$. 

The methods are to generalise the highly geometric fully augmented links in $S^3$ to lie on surfaces other than $S^2\subset S^3$. This will likely be of interest in its own right. Since their appearance in the appendix by Agol and Thurston in a paper of Lackenby~\cite{Lackenby:Vol}, fully augmented links have contributed a great deal to our understanding of the geometry of many knot and link  complements with diagrams that project to $S^2$. For example they have been used to bound volumes~\cite{FKP:DehnFilling} and cusp shapes~\cite{Purcell:CuspShapes}, give information on essential surfaces~\cite{BlairFuterTomova}, crosscap number~\cite{KalfagianniLee}, and short geodesics~\cite{Millichap}.

Such links on $S^2$ are amenable to study via hyperbolic geometry because their complements are hyperbolic and contain a pair of totally geodesic surfaces meeting at right angles: a projection surface, coloured white, and a disconnected \emph{shaded surface} consisting of many 3-punctured spheres; see~\cite{Purcell:AugLinksSurvey}. While essential 3-punctured spheres are geodesic in any hyperbolic 3-manifold, the white projection surface does not remain geodesic when generalising to links on surfaces other than $S^2$. However, using machinery from circle packings and Kleinian groups, we are able to construct links with a geometry similar to the projection surface.
We note other vey recent generalisations of fully augmented links to lie in thickened surfaces, due to Adams \emph{et al} \cite{adamsetal:2021Generalized}, Kwon~\cite{kwon2020fully}, and Kwon and Tham~\cite{kwon2020hyperbolicity}. We work within a different manifold, as follows.

Given a compact 3-manifold $\overline{M}$ with a single boundary component, the \emph{double of $\overline{M}$}, denoted $D(\overline{M})$ is the closed manifold obtained by gluing two copies of $\overline{M}$ by the identity along $\bdy \overline{M}$. 
The first main result of this paper is the following. 

\begin{theorem}\label{Thm:MainDouble}
Let $M$ be a geometrically finite hyperbolic 3-manifold of infinite volume that is homeomorphic to the interior of a compact manifold $\overline M$ with a single boundary component. Then there exists a sequence $M_n$ of finite volume hyperbolic manifolds that are knot complements in $D(\overline M)$, such that $M_n$ converges geometrically to $M$.

Moreover, the method is constructive: we construct for $p\in M$ and any $R>0$ and $\epsilon>0$ a fully augmented link complement $M_{\epsilon,R}$ in $D(\overline M)$ with a basepoint $p_{\epsilon,R}$ such that the metric ball $B(p_{\epsilon,R},R)\subset M_{\epsilon,R}$
is $(1+\epsilon)$-bilipschitz to the metric ball $B(p,R) \subset M$.
Performing sufficiently high Dehn filling along the crossing circles of the fully augmented link yields a knot complement, where the Dehn filling slopes can also be determined effectively, so that the resulting knot complement contains a metric ball that is $(1+\epsilon)^2$-bilipschitz to $B(p,R)$. 
\end{theorem}

We prove \refthm{MainDouble} by first proving the theorem in the convex cocompact case. In \refsec{GFtoCC}, we extend the result to the geometrically finite case. 

The density theorem states that any hyperbolic 3-manifold $M$ with finitely generated fundamental group is the algebraic limit of a sequence of geometrically finite hyperbolic 3-manifolds; see Ohshika~\cite{ohshika2011realising} and Namazi and Souto~\cite{namazi2012nonrealizability}. Namazi and Souto proved a strong version of this theorem~\cite[Corollary~12.3]{namazi2012nonrealizability}: that in fact, the sequence can be chosen such that $M$ is also the geometric limit. Thus an immediate corollary of \refthm{MainDouble} is the following.

\begin{corollary}\label{Cor:Density}
   Let $M$ be a hyperbolic 3-manifold of infinite volume which is homeomorphic to the interior of a compact manifold $\overline M$ with a single boundary component. Then there exists a sequence $M_n$ of finite volume hyperbolic manifolds that are knot complements in $D(\overline M)$, such that $M_n$ converges geometrically to $M$.
\end{corollary}

\subsection{Acknowledgements}
This work was supported in part by grants from the Australian Research Council, particularly FT160100232, and DP210103136.

\section{Background}\label{Sec:Background}

In this section we review definitions and results that we will need for the construction, particularly terminology and results in Kleinian groups and their relation to hyperbolic 3-manifolds. Further details are contained, for example, in the books \cite{marden2007outer} and \cite{kapovich2001hyperbolic}. 
	
\subsection{Kleinian Groups}

Recall that the ideal boundary of $\mathbb H^3$ is homeomorphic to $S^2$, which can be viewed as the Riemann sphere, and that the group of isometries $\text{Isom}(\mathbb{H}^3)$ corresponds to the group of M\"obius transformations acting on the boundary. We mostly consider orientation preserving M\"obius transformations here, which may be viewed as elements in $\PSL(2,\CC)$.

A discrete subgroup of $\PSL(2,\CC)$ is called a \emph{Kleinian group}. 

\begin{definition}
A point $x \in S^2$ is a \emph{limit point} of a Kleinian group $\Gamma$ if there exists a point $y \in S^2$ such that $\lim_{n \to \infty} A_n(y) = x$ for an infinite sequence of distinct elements $A_n \in \Gamma$. The \emph{limit set} of $\Gamma$ is 
$\Lambda(\Gamma) = \{ x \in  S^2 \mid x \text{ is a limit point of $\Gamma$} \}.$

The \emph{domain of discontinuity} is the open set $\Omega(\Gamma) = S^2 \setminus \Lambda(\Gamma)$. 
This set is sometimes called the \emph{ordinary set} or \emph{regular set}. 
\end{definition}
	
A Kleinian group $\Gamma$ is often studied by its quotient space:
\[
\mathcal{M}(\Gamma) = (\mathbb{H}^3 \cup \Omega(\Gamma)) / \Gamma
\]
If $\Gamma$ is torsion-free, then $\mathcal{M}(\Gamma)$ is an oriented manifold with possibly empty boundary $\bdy \mathcal{M}(\Gamma)=\Omega(\Gamma) / \Gamma$. The interior $\interior(\mathcal{M}(\Gamma)) = \mathbb{H}^3 / \Gamma$ has a complete hyperbolic structure, since its universal cover is $\mathbb{H}^3$. The fundamental group of $\interior(\mathcal{M}(\Gamma))$ is isomorphic to $\Gamma$.
By Ahlfors' finiteness theorem~\cite{ahlfors1964finitely, ahlfors1965finitelycorrection}, if $\Gamma$ is a finitely generated torsion-free Kleinian group, then $\Omega(\Gamma) / \Gamma$ is the union of a finite number of compact Riemann surfaces with at most a finite number of points removed.
The boundary $\partial \mathcal{M}(\Gamma)=\Omega(\Gamma) / \Gamma$ endowed with this conformal structure is called the \emph{conformal boundary} of $\mathcal{M}(\Gamma)$. The Teichm\"uller space $\mathcal{T}(\partial \mathcal{M}(\Gamma))$ is the product the Teichm\"uller spaces $\mathcal{T}(S_i)$ where the $S_i$ form the components of $\partial \mathcal{M}(\Gamma)$.

In fact, the conformal boundary $\partial \mathcal{M}(\Gamma)$  has a \emph{projective structure}, since it is locally modelled on $(\widehat{\CC},\PSL(2,\CC))$. A (projective) \emph{circle} on $\partial \mathcal{M}(\Gamma)$ is a homotopically trivial, embedded $S^1\subset \partial \mathcal{M}(\Gamma)$ whose lifts to $\Omega(\Gamma)$ are circles on $S^2$.
	
\begin{definition} 
  Let $\Gamma$ be a Kleinian group and let $D$ be an open disk in $ \Omega(\Gamma)$ whose boundary is a circle $C$ on $S^2$. The circle $C$ determines a hyperbolic plane in $\mathbb H^3$. Denote by $H(D)\subset \mathbb{H}^3$ the closed half-space bounded by this plane that meets $D$. The \emph{convex hull} of $\Lambda$ is the relatively closed set
  \begin{equation*}
    CH(\Gamma) = \mathbb H^3 - \bigcup_{D \subset \Omega(\Gamma)} H(D).
  \end{equation*}
  The \emph{convex core of $\mathcal{M}(\Gamma)$} is the quotient 
  \begin{equation*}
    CC(\Gamma) = CH(\Gamma) / \Gamma \subset \interior\left (\mathcal{M}(\Gamma) \right).
  \end{equation*}
\end{definition}

\begin{definition}
A finitely generated Kleinian group $\Gamma$ for which the convex core $CC(\Gamma)$ has finite volume is called \emph{geometrically finite}.
  
If the action of $\Gamma$ on $CH(\Gamma)$ is cocompact, then $\Gamma$ is said to be \emph{convex cocompact}.

A hyperbolic 3-manifold is called \emph{geometrically finite} (resp. \emph{convex cocompact}), if  it is isometric to $\HH^3/\Gamma$ for a geometrically finite (resp.\ convex cocompact) $\Gamma$. 
\end{definition}

If $\Gamma$ is convex cocompact and torsion-free, then it follows that $\partial \mathcal{M}(\Gamma)$ is a (possibly disconnected) compact Riemann surface without punctures.

There are several equivalent definitions of a geometrically finite manifold in 3-dimensions; see Bowditch~\cite{Bowditch} for a discussion. For example, we will also use the following, which follows from the proof in~\cite{Bowditch} that GF5 is equivalent to GF3, in Section~4 of that paper. 

\begin{theorem}[Bowditch~\cite{Bowditch}]\label{Thm:Bowditch}
The torsion-free Kleinian group $\Gamma$ is geometrically finite if and only if there a finite sided fundamental domain $\calF(\Gamma)\subset\HH^3$ for the action of $\Gamma$ on $\HH^3$, with sides of $\calF(\Gamma)$ consisting of geodesic hyperplanes.
\end{theorem}

If $CC(\Gamma)$ is compact, it must also have finite volume, and so convex cocompact manifolds are geometrically finite. However, geometrically finite manifolds may also contain cusps.
Marden showed that a torsion-free Kleinian group $\Gamma$ is geometrically finite if and only if $\calM(\Gamma)$ is compact outside of horoball neighbourhoods of finitely many rank one and rank two cusps~\cite{marden1974geometry}. The rank one cusps correspond to pairs of punctures on $\bdy\calM(\Gamma)$. 
	
\subsection{The Quasiconformal Deformation Space} 
Consider a finitely generated, discrete subgroup $\overline{\Gamma}$ of $\text{Isom}(\mathbb{H}^3)$
such that the normal subgroup (of index at most two) $\Gamma :=\overline{\Gamma}\cap \PSL_2(\mathbb{C})$
is torsion-free.
A representation $\rho\from \overline{\Gamma}\rightarrow \text{Isom}(\mathbb{H}^3)$ is \emph{a quasiconformal deformation} of $\overline{\Gamma}$, if there is a (orientation-preserving) $K$--quasiconformal homeomorphism $h\from S^2\rightarrow S^2$ for some $K\geq 1$, such that we have
\[
\rho(\gamma)=h_*(\gamma):=h\circ \gamma\circ h^{-1}:S^2\rightarrow S^2 \quad \text{ for all }\gamma\in \overline{\Gamma}.
\]
(We shorten $K$--quasiconformal homeomorphism to $K$-qc homeomorphism below.)

\begin{definition}
The \emph{quasiconformal deformation space $QC(\overline{\Gamma})$} of $\overline{\Gamma}$ is defined as
\[
QC(\overline{\Gamma}):=\{\rho \mid \rho \text{ is a quasiconformal deformation of } \overline{\Gamma} \}/\PSL_2(\mathbb{C}).
\]
It can be endowed with a Teichm\"uller metric given by
\[
d_T([\rho],[\rho']):=\inf\{ \log K \mid \exists \phi \text{ $K$-qc homeomorphism with }\rho=\phi\circ \rho'\circ\phi^{-1}\}.
\]
We will always endow $QC(\overline{\Gamma})$ with the topology induced by this metric.
\end{definition}

Now let $\Gamma$ have index two in $\overline{\Gamma}$. Then the extension $\Gamma\subset \overline{\Gamma}$ amounts to an orientation-reversing isometric involution $\sigma$ on $\mathcal{M}(\Gamma)$, as follows. The space $\mathcal{M}(\overline{\Gamma})$ is a possibly non-orientable orbifold with boundary $\bdy\calM(\overline{\Gamma})$. The orbifold $\calM(\overline{\Gamma})$ can be recovered as $\mathcal{M}(\Gamma)/\sigma$. In particular, $\bdy\calM(\overline{\Gamma})$ is given by the quotient $\partial \mathcal{M}(\Gamma)/\sigma$. 
Conversely, the Riemann surface double of the Klein surface $\partial \mathcal{M}(\overline{\Gamma})$ yields $\partial \mathcal{M}(\Gamma)$ identified by $\sigma$.
Note that by passing to the Riemann surface double, we obtain a continuous map $j\from\mathcal{T}(\partial \mathcal{M}(\overline{\Gamma}))\rightarrow \mathcal{T}( \partial \mathcal{M}(\Gamma))$ and that restricting gives a natural inclusion map $QC(\overline{\Gamma})\rightarrow QC(\Gamma)$ for any $\Gamma\subset \overline{\Gamma}$.

The first part of the following theorem follows from work of Bers~\cite{BersKleinian}, Kra~\cite{KraKleinian} and Maskit~\cite{MaskitSelfmaps} when restricting to torsion-free Kleinian groups.

\begin{theorem}\label{Thm:AhlforsBers}
Let $\Gamma$ be a torsion-free finitely generated Kleinian group. Then there is a continuous map $\beta\from \calT(\bdy\calM(\Gamma))\to QC(\Gamma)$ given by associating to a marked conformal structure on $\partial\mathcal{M}(\Gamma)$ the  corresponding quasiconformal deformation of $\Gamma$.

Analogously, if $\overline{\Gamma}\subset \text{Isom}(\mathbb{H}^3)$ is such that $\Gamma=\overline{\Gamma}\cap \PSL_2(\mathbb{C})$, then the composition $\beta\circ j$ is a continuous map $\mathcal{T}(\partial \mathcal{M}(\overline{\Gamma}))\rightarrow QC(\overline{\Gamma})\subset QC(\Gamma)$.
\end{theorem}

\begin{proof}
We recall a proof of the first part given by Kapovich \cite[p.187]{kapovich2001hyperbolic} and then show that the second part follows by the same argument; compare also \cite[section~8.15]{kapovich2001hyperbolic}.

Consider elements in $\mathcal{T}(\partial \mathcal{M}(\Gamma))$ as  equivalence classes $[f\from X\to Y]$ of quasiconformal maps defined on the conformal boundary $X:=\partial \mathcal{M}(\Gamma)$ of the hyperbolic 3-manifold associated to $\Gamma$. 
Such a quasiconformal map $f$ induces a Beltrami differential $\mu$ on $X$, which lifts to a Beltrami differential $\mu'$ on $\Omega(\Gamma)$ that is invariant under the action of $\Gamma$. 
Extending $\mu'$ by $0$ yields a $\Gamma$-invariant Beltrami differential $\overline{\mu}$ defined globally on $S^2$. 
Solving the Beltrami equation for $\overline{\mu}$ yields a quasiconformal homeomorphism $h\from S^2\to S^2$ with $K(h)=K(f)$; it conjugates each $\gamma\in \Gamma$ to a Möbius transformation since $\gamma^*\overline{\mu}=\overline{\mu}$. Thus $h$ gives a representation $h_*\from \Gamma\to \text{Isom}(\HH^3)$ via $\gamma\mapsto h\circ \gamma\circ h^{-1}$. 
We set $\beta([f\from X\to Y])=h_*$. 
This map $\beta$ is well-defined, since equivalent marked Riemann surfaces yield the same conjugacy class of representations of $\Gamma$ by Sullivan's rigidity theorem. Moreover, it follows that $\beta$ is distance non-decreasing, since $K(h)=K(f)$ for any fixed marking surface $X$; in particular $\beta$ is continuous.

If now $\overline{\Gamma}\subset \text{Isom}(\mathbb{H}^3)$ is such that $\Gamma=\overline{\Gamma}\cap 
\PSL_2(\mathbb{C})$, then elements in $\mathcal{T}(\partial \mathcal{M}(\overline{\Gamma}))$ can be viewed as equivalence classes of equivariant quasiconformal maps $f\from (X,\sigma)\rightarrow  (Y,\sigma_Y)$ (defined on $(X,\sigma)$ associated to $\overline{\Gamma}$) up to equivariant isotopies. Such a map $f$ induces a $\sigma$-invariant Beltrami differential $\mu$ on $X$. 
As before, $\mu$ lifts and extends to $\overline{\mu}$ on $S^2$, which is now $\overline{\Gamma}$-invariant. If $\overline{h}$ solves the Beltrami-equation for $\overline{\mu}$, then it conjugates $\overline{\Gamma}$ to a representation $\overline{h}_*$ of $\overline{\Gamma}$; in other words, it yields a quasiconformal deformation $[\overline{h}_*]\in QC(\overline{\Gamma})$ of $\overline{\Gamma}$ which restricts to $[h_*]\in QC(\Gamma)$. Since the map $j$ obtained by forgetting the involutions is continuous, the claimed result follows.
\end{proof}

\subsection{Geometric convergence of 3-Manifolds}
	
In this section we will discuss what it means for hyperbolic 3-manifolds to converge geometrically. Background can be found in \cite{benedetti2012lectures, marden2007outer, canary2006notesonnotes, kapovich2001hyperbolic}.

Let $B_R(O)$ denote the hyperbolic ball of radius $R$ centered at an origin $O\in \mathbb{H}^3$.
Fix such an origin together with a frame in its tangent space (still simply denoted by $O$). Then hyperbolic manifolds with framed basepoints are in bijective correspondence with torsion-free Kleinian groups: A hyperbolic manifold with framed basepoint $(M,p)$ corresponds to the unique torsion-free Kleinian group $\Gamma$ such that there is an isometry $M\rightarrow \mathbb{H}^3/\Gamma$ taking the framed basepoint $p$ to the image of $O$ in $\mathbb{H}^3/\Gamma$. 
Under this correspondence a change of framed basepoint corresponds to conjugation of the Kleinian group. We denote the hyperbolic manifold with framed basepoint corresponding to $\Gamma$ by $(\HH^3/\Gamma,O)$.

\begin{definition}
For $i=1,2$ let $(N_i,p_i)=(\mathbb{H}^3/\Gamma_i,O)$ be two hyperbolic manifolds with framed basepoints. We say that $(N_2,p_2)$ is \emph{$(\varepsilon,R)$-close} to $(N_1,p_1)$, if there is a $(1+\varepsilon)$-bilipschitz embedding $\tilde{f}\from \mathbb H^3 \supset B_{R}(O)\to \mathbb{H}^3$ such that
\begin{itemize}
\item $\tilde{f}$ is $\varepsilon$-close in $C^0$ to the inclusion, that is $d_{C^0}(\tilde{f}, id_{\mathbb H^3}\vert_{B_R(O)})\leq \varepsilon$ and
\item $\tilde{f}$ descends to an embedding $f\from N_1\supset B_R(O)/\Gamma_1\rightarrow N_2$.
\end{itemize} 
\end{definition}
\begin{definition}\label{Def:geometric_conv}
A sequence of hyperbolic manifolds with framed basepoints $(M_k,p_k)$ is said to \emph{converge geometrically} to $(M,p)$, if for all $\varepsilon,R>0$, there is $k_0\in \NN$ such that for $k\geq k_0$ we have $(M_k,p_k)$ is $(\varepsilon,R)$-close to $(M,p)$.
Further, we say that a sequence of hyperbolic manifolds $M_k$ \emph{converges geometrically} to a hyperbolic manifold $M$ if for some (or equivalently\footnote{Note that if $(M_n,p_n)=(\HH^3/\Gamma_n,O)$ converges to $(M,p)=(\HH^3/\Gamma, O)$ and $p'$ is another framed basepoint on $M$ corresponding to the image of $O'$ in $\HH^3$, then $(\HH^3/\Gamma_n,O')$ converges to $(M,p')$.}, any)
framed basepoint $p$ on $M$ there are framed basepoints $p_k$ on $M_k$ such that $(M_k,p_k)$ converges geometrically to $(M,p)$. Also, a sequence of embeddings $f_k:M\rightarrow M_k$ \emph{establishes geometric convergence of $M_k$ to $M$}, if for any framed basepoint $p$ of $M$ and any $(\varepsilon,R)$ the (lifts of the) maps $f_k$ show that $(M_k,f_k(p))$ is $(\varepsilon,R)$-close to $(M,p)$ for $k$ sufficiently large.

\end{definition}

\begin{remark}
  A sequence of framed hyperbolic manifolds with framed basepoints $(M_k,p_k)$ converges geometrically to $(M,p)$, if and only if the corresponding torsion-free Kleinian groups $\Gamma_k$ converge to $\Gamma$ in the Chabauty topology.
  
Indeed, the proof of Theorem~E.1.14 in \cite{benedetti2012lectures} adapts to show that geometric convergence
of hyperbolic manifolds with framed basepoints in the sense of 
Definition~\ref{Def:geometric_conv} implies the convergence of the associated Kleinian groups, even though we do not assume $\tilde{f}(0)=0$ or convergence in $C^\infty$. On the other hand, geometric convergence of hyperbolic manifolds with framed basepoints in the sense of \cite[Section~E.1]{benedetti2012lectures} (or, by Theorem~E.1.14 in \cite{benedetti2012lectures}, Chabauty convergence of torsion-free Kleinian groups) implies geometric convergence in the sense of Definition~\ref{Def:geometric_conv}. 
\end{remark}

\subsection{Controlled equivariant extensions}

We say a quasiconformal homeomorphism $\phi\from S^2\to S^2$ conjugates a Kleinian group $\Gamma_1$ into a Kleinian group $\Gamma_2$ if the prescription $\gamma\mapsto \phi\circ \gamma\circ \phi^{-1}$ defines a group isomorphism $\phi\from \Gamma_1\to \Gamma_2$.

The following result is from McMullen~\cite[Corollary~B.23]{McMullenbookrenormal}.

\begin{theorem}[Visual extension of qc conjugation]\label{Thm:visual}
Suppose $\phi\from \partial \mathbb H^3 \to \partial \mathbb H^3 $ is a $K$-quasiconformal homeomorphism  conjugating $\Gamma_1$ into $\Gamma_2$. Then the map $\phi$ has an extension to an equivariant $K^{3/2}$-bilipschitz diffeomorphism $\Phi$ of $\mathbb H^3$. In particular the manifolds $\mathcal M (\Gamma_1)$ and $\mathcal M (\Gamma_2)$ are diffeomorphic.
\end{theorem}

Strictly speaking, according to the conclusion of \cite[Corollary~B.23]{McMullenbookrenormal}, the map $\Phi$ is an equivariant ``$K^{3/2}$-quasi-isometry''. By~\cite[A.2~p.186]{McMullenbookrenormal}, this means that the extension $\Phi$ is an equivariant Lipschitz map whose differential is bounded by $K^{3/2}$. But $\Phi$ arises from the visual extension of the Beltrami Isotopy Theorem~B.22, which is obtained by integrating a smooth vector field Theorem~B.10; thus $\Phi$ is smooth. Since Corollary~B.23 also applies to the inverse map $\phi^{-1}$ and associates to it the map $\Phi^{-1}$ (by visually extending the reverse Beltrami isotopy), we can conclude that $\Phi$ is actually a $K^{3/2}$-bilipschitz diffeomorphism.

\begin{corollary}\label{Cor:close}
  Let $\varepsilon>0$ and $R>0$. There is $\delta>0$ such that if $\phi\from \partial \mathbb H^3 \to \partial \mathbb H^3 $ is a $(1+\delta)$-quasiconformal homeomorphism fixing 
$0,1,\infty$ and conjugating a torsion-free Kleinian group $\Gamma$ to $\Gamma_\phi$, then its visual extension $\Phi$ establishes that $(\mathbb H^3/\Gamma_\phi,p_\phi)$ is $(\varepsilon,R)$-close to $(\HH^3/\Gamma,p)$. Here both framed basepoints $p, p_\phi$ are induced by the framed basepoint $O$ in $\HH^3$.
\end{corollary}

\begin{proof}
As seen in the proof of~\cite[Theorem~B.21,~B.22]{McMullenbookrenormal}, the visual extension $\Phi\from \HH^3\cup \partial \HH^3\rightarrow \HH^3\cup \partial \HH^3$ of a $K$-quasiconformal homeomorphism extends by reflection across $ \partial \HH^3$ further to a $K^{9/2}$-quasiconformal map $\overline{\Phi}\from S^3\rightarrow S^3$ fixing $0,1,\infty$ on the equatorial sphere $\partial \HH^3\subset S^3$.
  
Now in any dimension $n\geq 2$ and for any $L\geq 1$, the collection of $L$-quasiconformal homeomorphisms $S^n\rightarrow S^n$ fixing three specified points forms a normal family~\cite[Theorem~6.6.33]{Gehring2017}. If $L=1$, this consists only of the identity~\cite[Theorem~6.8.4]{Gehring2017}.

It follows that for $K=1+\delta$ close to $1$, the visual extension of a $K$-quasiconformal homeomorphism $\phi$ is a homeomorphism $\Phi$ of $\HH^3\cup \partial \HH^3$ that is $C^0$-close to the identity. In particular, given $R,\varepsilon>0$ there is $\delta>0$, such that the visual extension $\Phi$ of any $(1+\delta)$-quasiconformal homeomorphism $\phi$ fixing $0,1,\infty$ is $\varepsilon$-close to the identity on $B_R(O)\subset \HH^3$. Furthermore, the quasiconformal homeomorphism $\phi$ is $(\Gamma,\Gamma_\phi)$-equivariant by construction of $\Gamma_\phi$ and thus so is $\Phi$ by Theorem~\ref{Thm:visual}.
Combining these statements yields the desired result.
\end{proof}
\subsection{Circle Packings} 
	
In this section we will define circle packings and present a few important results relating to them. For more information see Stephenson~\cite{stephenson2005introduction}. We will eventually use circle packings to glue 3-manifolds and obtain our desired knot and link complements.

\begin{definition}\label{CirclePackingConformal}
Let $\Gamma$ be a torsion-free convex cocompact Kleinian group and recall that its conformal boundary $\partial \mathcal{M}(\Gamma)$ has a natural projective structure. Let $V$ be a triangulation of $\partial \mathcal{M}(\Gamma)$.

A \emph{circle packing} on $\partial \mathcal{M}(\Gamma)$ with nerve $V$ is 
a collection $P = \{c_v \vert\, v \text{ vertex of } V\}$ of (projective) circles on $\partial \mathcal{M}(\Gamma)$ bounding discs with disjoint interiors such that
\begin{enumerate}
\item each circle $c_v$ is centered $v$,
\item two circles $c_u, c_v$ are tangent if and only if $\langle u, v \rangle$ is an edge in $V$, and
\item three circles $c_u, c_v, c_w$ bound a positively oriented curvilinear triangle in $\partial \mathcal{M}(\Gamma)$ if and only if $\langle u, v, w \rangle$ form a positively oriented face of $V$.
\end{enumerate}
More generally, if $V$ is just a connected graph embedded in $\partial \mathcal{M}(\Gamma)$, we say that a collection of (projective) circles satisfying the first two conditions form a \emph{partial circle packing} with nerve $V$. 

Equivalently, we can consider locally finite, $\Gamma$-equivariant (partial) circle packings of $\Omega(\Gamma)$ obtained as lifts of (partial) circle packings on $\partial \mathcal{M}(\Gamma)$. 
\end{definition}

See \reffig{CirclePacking} for an example of a circle packing. 

\begin{figure}
  \centering
  \includegraphics{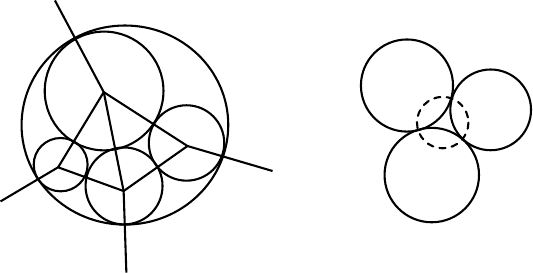}
  \caption{Left: Example of a circle packing with its nerve, the edges going out all meet at the vertex at $\infty$. Right: Three circles in a circle packing along with their dual circle, drawn with a dashed line.}
  \label{Fig:CirclePacking}
\end{figure}

\begin{definition}
Let $P$ be a circle packing with nerve $V$ and let $c_1,c_2,c_3 \in P$ be circles corresponding to a triangle in $V$. The curvilinear triangle bounded by these circle is called an \emph{interstice}.  There is a unique circle $c^{(1,2,3)}$ orthogonal to the circles $c_1,c_2,c_3$, intersecting them at their points of tangency. The collection of all such circles corresponding to each triangle in $V$ we will denote $P^*$ and we will call the \emph{dual (partial) circle packing} of $P$, see \reffig{CirclePacking}. Note that the nerves of $P$ and $P^*$ are duals as graphs on the surface.
\end{definition}

Work of Brooks~\cite{brooks1986circle} shows that convex cocompact hyperbolic 3-manifolds admitting a circle packing on the conformal boundary are abundant, in the following sense.

\begin{theorem}[Circle packings approximate]\label{Thm:BrooksKleinian}
Let $M=\HH^3/\Gamma$ be a convex cocompact hyperbolic 3-manifold.
Then, for every $\epsilon>0$, there is an $e^\epsilon$-quasiconformal homeomorphism $\phi$ fixing $0,1,\infty$, conjugating $\Gamma$ to $\Gamma_\epsilon$ such that the conformal boundary of $M_\epsilon=\HH^3/\Gamma_\epsilon$ admits a circle packing.

Moreover, the process is constructive: the proof constructs the circle packing. Additionally, for fixed $r>0$, we may ensure none of the circles in the circle packing and none of the triangular interstices have diameter larger than $r$. Here we identify $\partial \HH^3$ with the unit sphere in the tangent space $T_O\HH^3$ at the framed basepoint $O$ in $\HH^3$.
\end{theorem}

This is essentially contained in Brooks' proof of~\cite[Theorem~2]{brooks1986circle}, but the statement of the theorem is different in Brooks' paper. In particular, there was no consideration of diameters there, and no worry about construction. We work through the proof below, highlighting the diameters and the constructive nature of the proof. 
\begin{proof}[Proof of \refthm{BrooksKleinian}]
We begin by choosing effective constants controlling the diameters of the circles and interstices, using a compactness argument. We may uniformise each closed surface component of $\Omega(\Gamma)/\Gamma$ by a component of $\Omega(\Gamma)$. Because $\Gamma$ is convex cocompact, hence geometrically finite, its action has a finite-sided fundamental domain $F$ by \refthm{Bowditch}, giving a finite-sided fundamental region for the action of $\Gamma$ on $\Omega(\Gamma)$. The fundamental region will have boundary consisting of vertices and edges, and will be compact.

We need to choose the circles to have bounded radii when seen from $O$ in $\HH^3$. To do so, it is convenient to look at hyperbolic space in the Poincar\'e ball model $\mathbb{B}^3$ with $O$ at the origin. Then circles of radius $r$ in the unit sphere of $T_O \HH^3$ correspond to circles of radius in $r$ the boundary of the Poincar\'e ball $\partial \mathbb{B}^3$.

For given $r>0$, pick a small $r_0\leq r/2$ such that any disk $D$ of radius $r_0$ meeting $F$ intersects, apart from $F$,
at most the immediate neighbouring fundamental domains to $F$ in $\Omega(\Gamma)$. Since $F$ is compact, it can be covered by finitely many open discs $D_i$ of radius $r_0$.
All translates of these $D_i$ are round disks; therefore the diameter of each translate $\gamma(D_i)$ is bounded in terms of their area. This implies that there are only finitely many translates $\gamma(D_i)$ whose diameter is larger than $r_0$. Indeed, otherwise there would be an infinite disjoint collection of such translates of diameter larger than $r_0$, but this is impossible since the area of $S^2$ is finite.
It follows that there are only finitely many translates $F_1,\ldots , F_k$ of $F$ that meet a translate $\gamma(D_i)$ whose diameter is larger than $r_0$.

Therefore we can pick $r_1\leq r_0$ such that for any disk $D$ of radius at most $r_1$ meeting $F$ the following holds: $D$ is contained in one of the $D_i$, and any translate of $D$ meeting $F_1, \ldots , F_k$ has diameter at most $r_0$. Note also that translates of $D$ \emph{not} meeting $F_1,\ldots,F_k$ automatically have diameter at most $r_0$ by construction.

Now pack $F$ with circles of radius at most $r_1$ by the following constructive process, similar to that of~\cite[Lemma~2.3]{HoffmanPurcell}.
First choose disjoint circles centred at vertices of $F$, taking their images under $\Gamma$ to ensure equivariance. Then take circles centred along edges, again ensuring translates under $\Gamma$ agree. Finally, take circles of radius at most $r_1$ with centres in the interior of the region. Extend this partial circle packing of $F$ to $\Omega(\Gamma)$ using the action of $\Gamma$, ensuring an equivariant packing.

This yields a $\Gamma$-equivariant partial circle packing of $\Omega(\Gamma)$ consisting of circles of diameter at most $r_0$ and with regions complementary to the circles consisting of polygonal interstices, with circular arcs as boundaries. At this point, additional circles of radius at most $r_1$ may be added to $F$; we add sufficiently many to obtain interstitial regions that are either triangles or quads of diameter at most $r_1$; see Brooks~\cite{Brooks:Schottky} or a more detailed exposition in Stewart \cite[Lemma~3.7]{Stewart:Thesis}. Finally, extend again $\Gamma$-equivariantly to obtain an equivariant partial packing of $\Omega(\Gamma)$ with circles of diameter at most $r_0$, all of whose interstitial regions are triangles or quads of diameter at most $r_0\leq r/2$.

Consider the group $\overline{\Gamma}$ generated by $\Gamma$ and all reflections across the circles in the packing. By \refthm{AhlforsBers}, the Teichm\"uller space of the complementary regions, which here are triangles and quads, maps continuously to the quasiconformal deformation space of $\overline{\Gamma}$ with its Teichm\"uller metric.
The triangular interstitial regions are conformally rigid. The quads have a Teichm\"uller space  homeomorphic to $\RR$.

Brooks shows in \cite{Brooks:Schottky} that there is an explicit homeomorphism $q$ from the Teichm\"uller space of a quad to $\mathbb{R}$ with the property that there is a full packing of a quad by finitely many circles if and only if $q(Q)$ is rational. 
Thus, arbitarily close to any quad $Q$ in the Teichm\"uller space of quads, there is another quad $Q'$ with $q(Q')$ rational.
Applying this simultaneously to all the quads complementary to the packing, we obtain arbitrarily close configurations where $q(Q')$ is rational for all quads $Q'$. We may uniquely pack circles into this quad. 

By \refthm{AhlforsBers}, for any $\epsilon>0$, we can thus quasiconformally deform the associated representation $\rho\from \overline{\Gamma}\to  \text{Isom}(\mathbb{H}^3)$ by an $e^\epsilon$-quasiconformal homeomorphism $h$, normalized to fix the points $0,1,\infty$, to obtain a new convex cocompact representation with image $\overline{\Gamma}_\epsilon$, whose complementary quads are all rational. See \reffig{CircleDeformation}.
\begin{figure}
  \centering
  \includegraphics{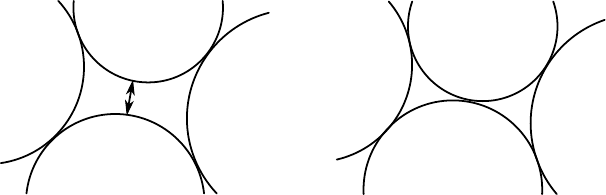}
  \caption{Once enough circles are added so a quad is sufficiently small, the space is deformed so that opposite circles become tangent. Doing this for every quad gives a circle packing.}
  \label{Fig:CircleDeformation}
\end{figure}

We need to ensure that the quasiconformal homeomorphism does not enlarge the diameters of circles and interstitial regions too much. Indeed, for any $K\geq 1$, the $K$-quasiconformal homeomorphisms of $S^2$ fixing $0,1,\infty$ form a normal family. Because we fix $0,1,\infty$, this normal family consists of only the identity map when $K=1$. Thus any sequence of $K_i$-quasiconformal homeomorphisms of $S^2$ fixing $0,1, \infty$ with $K_i\rightarrow 1$ converges to the identity map on $S^2$; compare the proof of Corollary~\ref{Cor:close}. 

Thus, while the $e^\epsilon$-quasiconformal deformation may enlarge some of the radii of the circles, provided $\epsilon$ is small enough, the resulting circles and interstitial regions will have diameter at most $r$.
\end{proof}

\begin{definition}
  Let $V$ be a graph. A \emph{dimer} on $V$ is a colouring of edges such that each face is adjacent to exactly one coloured edge.
\end{definition}

\begin{lemma}\label{Lem:AddCirclesDimer}
Let $\Gamma$ be a torsion-free convex cocompact Kleinian group and let $P$ be a $\Gamma$-equivariant circle packing of $\Omega(\Gamma)$ with nerve $V$.

 Then there exists a circle packing $\overline{P}$ with nerve $\overline{V}$ such that $V \subset \overline{V}$ and $\overline{V}$ admits a dimer.
Further, the maximal diameter of circles and interstitial regions of $\overline{P}$ in $\Omega(\Gamma)$ does not exceed that of $P$.
\end{lemma}

\begin{proof}
We define the circle packing $\overline{P}$ by adding the unique circle to each triangular interstice in $P$ which is tangent to all three circles. The effect on the nerve is to add a vertex to the interior of each triangle of $V$, and connect by three edges to the existing vertices of $V$, subdividing each triangle into three triangles to form $\overline{V}$. Then each triangle in $\overline{V}$ has exactly one edge coming from $V$. Colour this edge. This gives a dimer on $\overline{V}$. Observe that because the action of $\Gamma$ takes triangular interstices to triangular interstices, the result is still equivariant with respect to $\Gamma$.
Observe that the diameter of circles and interstitial regions at most decrease with this procedure.
\end{proof}

In general there are multiple ways to add circles to a circle packing so that the result admits a dimer. The strength of the above its that it works for any starting circle packing and is simple to execute.
	
\section{Construction}\label{Sec:Construction}

In this section, we construct the links of the main theorem.

\subsection{Scooped Manifolds}
				
\begin{definition}
Let $M = \mathbb H^3 / \Gamma$ be a convex cocompact hyperbolic 3-manifold.
Further assume that $\bdy\calM(\Gamma)=\Omega(\Gamma)/\Gamma$ admits a circle packing $P$ with dual packing $P^*$; then on $\Omega(\Gamma)$ there is a corresponding equivariant circle packing $\widetilde{P}$ with dual packing $\widetilde{P}^*$.
For the circles $c_i$ in $\widetilde{P}$ on $\Omega(\Gamma)$, there are pairwise disjoint associated open half spaces $H(c_i) \subset \HH^3$ which meet the conformal boundary $\bdy\HH^3$ at the interior of $c_i$.
We then define the \emph{scooped manifold} $M_P$ to be the manifold formed by removing the half spaces associated with circles in $\widetilde{P}$ and its dual $\widetilde{P}^*$, and taking the quotient under $\Gamma$:
\begin{equation*}
  M_P = \HH^3 - \bigcup_{c \in \widetilde{P},\widetilde{P}^*} H(c) /\Gamma.
\end{equation*}
\end{definition}
	
The boundary of $M_P$ consists of hyperbolic ideal polygons whose faces come from $\partial H(c), c \in P$ and $\partial H(c^*), c^* \in P^*$, and edges come from the intersection of $\partial H(c)$ and $\partial H(c^*)$. Note $M_P$ is a manifold with corners whose interior is homeomorphic to $M$.

\begin{lemma}\label{Lem:ScoopedContainingBall}
Let $M = \HH^3/\Gamma$ be a convex cocompact hyperbolic 3-manifold and $O\in \HH^3$.
Then for any $\epsilon>0$, there exists a $e^\epsilon$-quasiconformal homeomorphism $\phi$ fixing $0,1,\infty$ conjugating $\Gamma$ to $\Gamma_\epsilon$ 
satisfying the following:
\begin{itemize}
\item The associated convex cocompact manifold $M_\epsilon = \HH^3/\Gamma_\epsilon$ admits a circle packing $P$ on its conformal boundary.
\item The metric ball $B(O,R)/\Gamma_\epsilon \subset M_\epsilon$
is completely contained in the corresponding scooped manifold $(M_\epsilon)_P$.
\item Further, we can extend $P$ to a circle packing $\overline{P}$ that admits a dimer as in \reflem{AddCirclesDimer}, so that $B(O,R)/\Gamma_\epsilon$ is still completely contained in the scooped manifold $(M_\epsilon)_{\overline{P}}$.
\end{itemize}
\end{lemma}

\begin{proof}
The construction of \refthm{BrooksKleinian} yields an $e^\epsilon$-quasiconformal homeomorphism fixing $0,1,\infty$, and giving $M_\epsilon$ with circle packing $P$ on its conformal boundary, where circles and triangular interstices have diameter at most $r$. For $r>0$ sufficiently small, we may ensure that the half-spaces $H(c)$ defined by the circles of $P$ and its dual $P^*$ have distance at least $2R$ from $O$ in $\HH^3$. 
Thus we have $B(O,R)/\Gamma_\epsilon \subset M_\epsilon - \cup_{c \in P} H(c)  = (M_\epsilon)_P$. 
  
Finally, using \reflem{AddCirclesDimer}, we can extend $P$ to a circle packing $\widetilde{P}$ which admits a dimer.
\end{proof}

\begin{proposition}\label{Prop:BoundaryMP}
Let $M$ be a convex cocompact hyperbolic 3-manifold. Further suppose that $\bdy \calM(\Gamma)$ admits a circle packing $P$ with nerve $K$ that has a fixed dimer. Then the scooped manifold $M_P$ has the following properties:
\begin{enumerate}
\item The faces on the boundary of $M_P$ can be checkerboard coloured, white and black.
\item The white faces consist of totally geodesic ideal polygons.
\item The black faces consist of totally geodesic ideal triangles. The dimer induces a pairing of the black faces, such that paired black faces share an ideal vertex.
\item The ideal vertices are all four valent.
\item The dihedral angle between faces on the boundary is $\pi /2$. 
\end{enumerate}
\end{proposition}

\begin{proof}
By the definition of scooped manifolds the boundary of $M_P$ consists of ideal geodesic polygons coming from the boundaries of the half spaces associated with circles in $P$ and $P^*$. The geodesic polygons coming from half spaces associated with circles in $P$ we colour white, while those coming from $P^*$ we colour black. Observe that the points of tangency of circles in $P$ and $P^*$ are the same, so these points of tangency form the ideal vertices of both the black and white faces. If $c \in P$ and $c^* \in P^*$ are circles such that $c \cap c^* \neq \varnothing$ then $c$ and $c^*$ intersect in exactly two points $u$ and $v$; these points of intersection correspond to ideal points on the boundary of $M_P$. There is an edge between $u$ and $v$ on $\partial M_P$ formed by $H(c) \cap H(c^*)$. This edge lies between the face corresponding to $H(c)$ which we have coloured white and $H(c^*)$ which we have coloured black. Since every edge on $\partial M_P$ occurs in this manner, we know that every edge lies between a black and white face. Thus we know that the colouring of the faces we have assigned gives a checkerboard colouring of the faces. The fact that ideal vertices are 4-valent follows from the fact that at each ideal vertex there are four circles which meet at this point: two from $P$ and two from $P^*$. Finally, since circles in $P$ and $P^*$ meet orthogonally, the dihedral angle at each edge must be $\pi/2$. 
		
To see that the black faces are triangles, observe that for every circle $c^* \in P^*$ we have by definition that $c^*$ meets exactly three points in $P$. These points are the ideal vertices on the black faces corresponding to the half space associated with $H(c^*)$.

\begin{figure}
  \centering
  \includegraphics{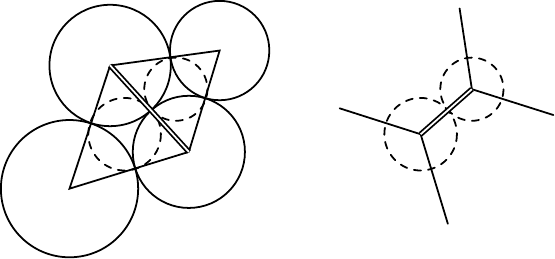}
  \caption{Left shows four circles in $P$, with two dashed circles in $P^*$. Part of the nerve of $P$ is shown on the left with the coloured edge from the dimer drawn with two lines. On the right we have the same two circles in $P^*$ along with the colouring of the associated part of the nerve of $P^*$.}
  \label{Fig:PairedBlackTri}
\end{figure}
		
Now we show how the black faces are paired. Let $K$ be the nerve of $P$, which has a dimer. Then in the dual graph $K^*$ of $K$, we can transfer the colouring of edges in $K$ to a colouring of edges in $K^*$, since edges are sent to edges. Note that $K^*$ is 3-valent since $K$ only consists of triangles. Since each face in $K$ is adjacent to exactly one coloured edge in the dimer, each vertex in $K^*$ is adjacent to exactly one coloured edge. This gives a pairing on the vertices in $K^*$ along this edge, which gives a paring of the circles in $P^*$. Thus each black face in $\partial M_P$ is paired to another black face. See \reffig{PairedBlackTri}.
\end{proof}

\begin{lemma}\label{Lem:RectangleVertices}
Let $M=\HH^3/\Gamma$ be a convex cocompact hyperbolic 3-manifold and suppose that $\bdy \calM(\Gamma)$ admits a circle packing $P$.  For each ideal vertex $v_i \in \{v_1, \dots, v_n\}$ of the scooped manifold $\bdy M_P$, there is a horoball neighbourhood $H_i$ such that the $H_i$ are pairwise disjoint, and $\bdy H_i \cap M_P$ is a Euclidean rectangle.
\end{lemma}	

\begin{proof}
Let $\{v_1, \dots , v_n\}$ be the collection of ideal vertices on $\partial M_P$. Note that there are two circles in $P$ and two circles in $P^*$ which meet tangentially at each $v_i$. Let $\widetilde{M_P}$ denote a lift of $M_P$ into $\mathbb H^3$ under a covering map, and $\widetilde{v_i}$ a single point in the corresponding lift of $v_i$ to $\partial \mathbb H^3$. Two circles of $P$ and two of $P^*$ lift to be tangent to $\widetilde{v_i}$. Let $\varphi$ denote a M\"obius transformation taking $\widetilde{v_i}$ to $\infty$. It takes the circles projecting to $P$ to a pair of parallel lines, and those projecting to $P^*$ to another pair of parallel lines meeting the first two orthogonally, hence forming a Euclidean rectangle. Then any horoball $H_h$ of height $h$ centred at $\infty$ in $\HH^3$ meets $\varphi(\widetilde{M_P})$ in $R_i\times (h,\infty)$, where $R_i$ is a Euclidean rectangle. This projects to a rectangular horoball neighbourhood of $v_i$. Finally, because there are only finitely many ideal vertices of $M_P$, we may choose the horoball about each vertex so that all horoballs are pairwise disjoint, as desired.
\end{proof}

\begin{lemma}\label{Lem:FiniteVolume}
Let $M=\HH^3/\Gamma$ be a convex cocompact hyperbolic 3-manifold and suppose that $\bdy \calM(\Gamma)$ admits a circle packing $P$. Then the scooped manifold $M_P$ has finite volume.
\end{lemma}

\begin{proof}
Let $\{H_1, \dots, H_n\}$ be pairwise disjoint horoballs, one for each ideal vertex of $M_P$, as in \reflem{RectangleVertices}. 
Then removing these horoballs and horoball neighbourhoods from $M_P$ yields a compact manifold with boundary consisting of finitely many boundaries of horoball neighbourhoods and Euclidean planes $H_i\cap M_P$, and finitely many hyperplanes $\bdy H(c)\cap M_P$, where $c\in P$ or $P^*$ is from the circle packing or its dual. This has finite volume.

Finally, the horoball neighbourhoods must have finite volume, since they are of the form $R_i\times [1,\infty)$ for $R_i$ a Euclidean rectangle, as in \reflem{RectangleVertices}. Thus $M_P$ has finite volume.
\end{proof}

\subsection{Building Link Complements}

In this section we describe how to build a hyperbolic link complement using a scooped manifold. The idea behind this construction is inspired by \emph{fully augmented links}, and their relation to circle packings on the sphere. The construction here generalises this by starting with circle packings on a surface of higher genus.

First, we define a generalisation of a fully augmented link. 
	
\begin{definition}\label{Def:FALNoTwists}
Let $M$ be a 3-manifold and let $\Sigma$ be an embedded surface of genus $g \ge 2$ in $M$. Then a link $L$ in a tubular neighborhood of $\Sigma$ consisting of components $K_1, \dots K_k$ and $C_1, \dots C_n$ is called a \emph{fully augmented link on} $\Sigma$ if it has the following properties.
\begin{enumerate}
\item $\coprod_{1\leq i\leq k} K_i$ is embedded in $\Sigma$.
\item $C_j$ bounds a disk $D_j$ in $M$ such that $D_j$ intersects $\Sigma$ transversely in a single arc, and $D_j$ meets the union $\coprod_{i} K_i$ in exactly two points, for $1 \leq j \leq n$.
\item A projection of $L$ to $\Sigma$ yields a 4-valent \emph{diagram graph} on $\Sigma$. We require this diagram to be connected. 
\end{enumerate}
The components $K_i$ are said to lie in the \emph{projection surface}, while the components $C_j$ are called \emph{crossing circles.}
\end{definition}

We may also add a half twist at crossing circles, corresponding to cutting along $D_j$ and regluing so that the two points of intersection of $\coprod_i K_i$ with $D_j$ are swapped. This is shown in \reffig{HalfTwist}.  
\begin{figure}
  \centering
  \includegraphics{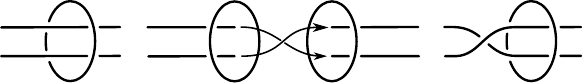}
  \caption{Shows how to cut and reglue at a crossing circle to add a half-twist.}
  \label{Fig:HalfTwist}
\end{figure}

\begin{definition}
The link resulting from adding a single half-twist at some or no crossing circles is also called a \emph{fully augmented link on a surface}, even though condition (1) in \refdef{FALNoTwists} is typically not satisfied anymore after such a half-twist. If the distinction is important, we will say that the link of \refdef{FALNoTwists} is a fully augmented link on a surface \emph{without half-twists}.
\end{definition}

Fully augmented links on surfaces can be quite complicated. A 3-dimensional example on a genus-2 surface is shown in \reffig{FAL3D}. 

\begin{figure}
  \includegraphics[scale=0.2]{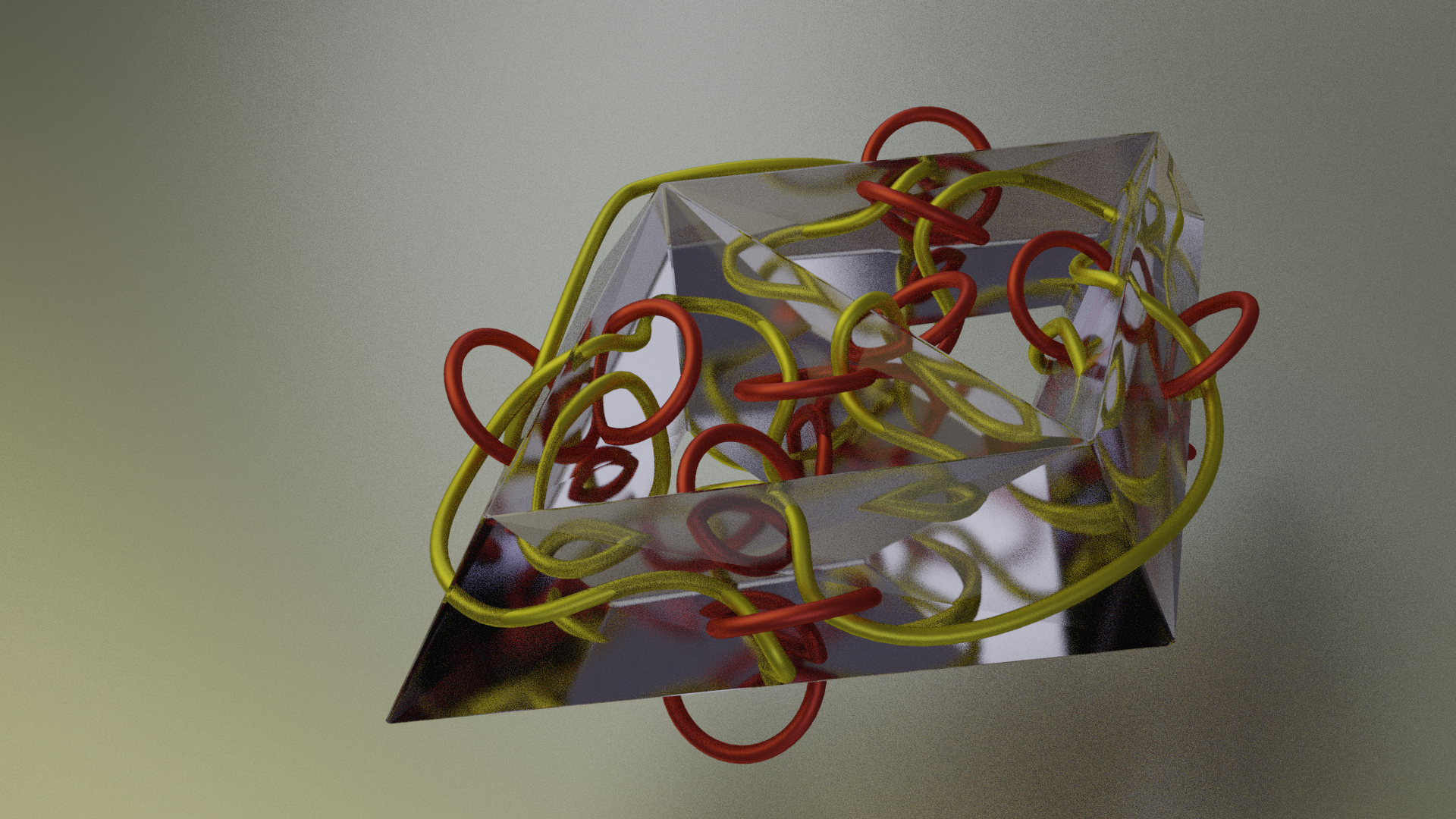}
  \caption{An example of a fully augmented link on a genus-2 surface, with crossing circles shown in red. This image was generated in Blender~\cite{blender}.}
  \label{Fig:FAL3D}
\end{figure}
	
\begin{definition}
Let $M$ be a manifold with boundary. The \emph{double} of $M$ is the manifold
\[
M \times \{0,1\} / \sim \quad \text{ where } (x,0) \sim (x,1) \text{ for all } x \in \partial M.
\]
We denote the double of $M$ by $\mathcal{D}(M)$.
\end{definition}

\begin{proposition} \label{Prop:DMNeverS3}
Let $M$ be an orientable compact manifold with connected boundary. Then the double of $M$ is not $S^3$, unless $\partial M$ is homeomorphic to $S^2$.
\end{proposition}
	
\begin{proof}
Let $M_1$ and $M_2$ denote the two copies of $M$ in the double of $M$, where $\text{int}(M_1) \cap \text{int} (M_2) = \varnothing$ and $\partial M_1 = \partial M_2$. Now for a point $x \in M_2$ let $\tilde {x}$ denote the same point in $M_1$, or if $x \in M_1$ then $\tilde{x}$ denotes the point in $M_2$. Then the map $r\from \mathcal \mathcal{D}(M) \to M_1$ defined by
\[
r(x) = \begin{cases}
  x \text{ if } x \in M_1, \\
  \tilde {x} \text{ if } x \in M_2
\end{cases}
\]
satisfies $r \mid_{M_1}$ is the identity. Moreover, $r$ is continuous since it is continuous on $M_1$ and $M_2$ and agrees on $M_1\cap M_2 = \bdy M_1$. Thus $r$ is a rectract of $\calD(M)$ onto $M_1$. It follows that the inclusion $M \hookrightarrow \calD(M)$ induces an injection $i_*\from \pi_1(M_1) \to \pi_1(\mathcal D (M))$.

On the other hand, $\pi_1(M_1)$ is nontrivial, since its abelianisation $H_1(M)$ has rank equal to half the rank of $H_1(\bdy M_1)$, which is $2g\geq 2$ unless $\partial M_1=S^2$; see~\cite[Lemmas~3.5, 3.6]{Hatcher:3MfldNotes}. 
Thus $\calD(M)$ is not $ S^3$ unless $\partial M=S^2$. 
\end{proof}

We are now ready to start our construction. 

\begin{construction}\label{Constr:FAL}
Let $M=\HH^3/\Gamma$ be a convex cocompact hyperbolic 3-manifold whose conformal boundary on $\bdy\calM(\Gamma)$ admits a circle packing $P$ with dimer.

By \refprop{BoundaryMP}, the boundary of the scooped manifold $M_P$ is checkerboard coloured black and white, with all black faces consisting of paired totally geodesic ideal triangles.

Form the scooped manifold $M_P$.
Take a second copy $M_P'$ of $M_P$ with the opposite orientation and identify each white face of $M_P$ with its copy in $M_P'$ via the identity map identifying these faces.

Black faces in $M_P$ are each paired in $M_P$ by the dimer, with the coloured edge of the dimer running over a pair of ideal vertices in the two triangles. Glue these paired ideal triangles by a hyperbolic isometry, folding over the ideal vertex meeting the dimer. Do the same for the paired black triangles in $M_P'$.
\end{construction}

\begin{theorem}\label{Thm:LinksInDouble}
Let $M = \HH^3/\Gamma$ be a convex cocompact hyperbolic 3-manifold. Suppose the conformal boundary $\bdy\calM(\Gamma)$ admits a circle packing with a dimer. Then Construction~\ref{Constr:FAL} above yields a finite volume hyperbolic 3-manifold $N$ that is the complement of a fully augmented link $L$ on $\bdy\calM(\Gamma)$ in $\calD(\calM(\Gamma))$, without half-twists. That is, $N = \calD(\calM(\Gamma))-L$.
\end{theorem}
	
\begin{proof}
Let $N$ denote the manifold obtained by the construction. There are three things we need to show: the construction gives a submanifold of $\calD(\calM(\Gamma))$, the result is homeomorphic to a fully augmented link complement in $\calD(\calM(\Gamma))$, and that it is a complete hyperbolic manifold of finite volume.

For ease of notation, we will denote $\calM(\Gamma)$ simply by $\calM$.
We start by showing that $N$ is a submanifold of $\calD(\calM)$. The definition of a scooped manifold gives a natural embedding of $M_P$  and $M_P^\prime$ in $\calD(\calM)$ such that $M_P \cap M_P' = \varnothing$. Under this embedding the ideal vertices of $M_P$ and $M_P'$ are identified and lie on $\Sigma = \bdy \calM = \bdy \calM'$ in $\calD(\calM)$.

By \reflem{RectangleVertices}, there is a collection of horoball neighbourhoods $H_i$ with boundaries meeting the ideal vertices in Euclidean rectangles $R_i$. By shrinking the $H_i$ if needed, we may assume that for each rectangle, 
the length of any side meeting a black triangle is $1/h$, for some fixed large $h$.
Let $\overline{M_P}$ denote the result of removing the horoballs $H_i$ from $M_P$. Thus $\overline{M_P}$ is a compact manifold with corners. Similarly form $\overline{M_P}'$ by removing identically sized horoball neighbourhoods from $M_P'$. 

Since the (black) truncated side lengths of $M_P$ are identical, we can glue truncated black triangles in $\overline{M_P}$ to their pair in $\overline{M_P}$ by hyperbolic isometry, and similarly for $\overline{M_P}'$. We may similarly glue truncated white faces in $\overline{M_P}$ to those in $\overline{M_P}'$ by isometry, because we will be truncating an identical amount in $M_P$ and its reflection.

Let $F \subset \partial \overline{M_P}$ be a truncated white face. Then there exists a projection $p\from F \to \Sigma$. Similarly, the corresponding truncated white face $F' \subset \partial \overline{M_P}'$ has an analogous projection $p' \from F' \to \Sigma$ such that $p(F) = p'(F')$. Both of these projections can be extended to isotopies of $\overline{M_P}$ and $\overline{M_P}'$ in $\calD(\calM)$. Since all such maps, for all white faces, correspond to isotopies, the manifold resulting from gluing the white faces is a submanifold of $\calD(\calM)$. 
				
Next we look at gluing pairs of truncated black triangles. Let $T_1$ and $T_2$ be two truncated black triangles in $\bdy \overline{M_P}$ that are paired by the dimer on $P$ across a vertex $v$, and let $R_v$ be the rectangle which truncates $v$. Similarly let $T_1'$ and $T_2'$ be the corresponding truncated triangles in $\bdy \overline{M_P}'$ with $R_v'$ the rectangle meeting them. After identifying the white faces, the non-truncated edges of $T_1$ and $T_1'$ will be identified, and similarly for $T_2$ and $T_2'$. Then after gluing white faces, $T_1 \cup T_1'$ and $T_2 \cup T_2'$ will correspond to a pair of spheres with three open disks removed. They are joined together via $R_v$ and $R_v'$: after we identify the white faces, the white edges of $R_v$ and $R_v'$ have been identified, forming a cylinder $A$. The black edges on the ends of this cylinder form one of the boundary components of both spheres $T_1 \cup T_1'$, $T_2 \cup T_2'$. See \reffig{GlueBlackFaces}.
		
\begin{figure}
  \centering
\begingroup%
  \makeatletter%
  \providecommand\color[2][]{%
    \errmessage{(Inkscape) Color is used for the text in Inkscape, but the package 'color.sty' is not loaded}%
    \renewcommand\color[2][]{}%
  }%
  \providecommand\transparent[1]{%
    \errmessage{(Inkscape) Transparency is used (non-zero) for the text in Inkscape, but the package 'transparent.sty' is not loaded}%
    \renewcommand\transparent[1]{}%
  }%
  \providecommand\rotatebox[2]{#2}%
  \newcommand*\fsize{\dimexpr\f@size pt\relax}%
  \newcommand*\lineheight[1]{\fontsize{\fsize}{#1\fsize}\selectfont}%
  \ifx\svgwidth\undefined%
    \setlength{\unitlength}{326.27670889bp}%
    \ifx\svgscale\undefined%
      \relax%
    \else%
      \setlength{\unitlength}{\unitlength * \real{\svgscale}}%
    \fi%
  \else%
    \setlength{\unitlength}{\svgwidth}%
  \fi%
  \global\let\svgwidth\undefined%
  \global\let\svgscale\undefined%
  \makeatother%
  \begin{picture}(1,0.2809173)%
    \lineheight{1}%
    \setlength\tabcolsep{0pt}%
    \put(0,0){\includegraphics[width=\unitlength,page=1]{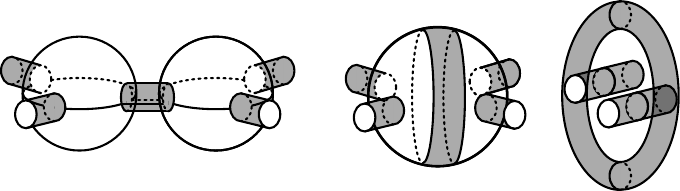}}%
    \put(0.62057305,0.25416969){\color[rgb]{0,0,0}\makebox(0,0)[lt]{\lineheight{1.25}\smash{\begin{tabular}[t]{l}$h_{1/2}(A)$\end{tabular}}}}%
    \put(0.20270857,0.17409985){\color[rgb]{0,0,0}\makebox(0,0)[lt]{\lineheight{1.25}\smash{\begin{tabular}[t]{l}$A$\end{tabular}}}}%
    \put(0.13800928,0.23522459){\color[rgb]{0,0,0}\makebox(0,0)[lt]{\lineheight{1.25}\smash{\begin{tabular}[t]{l}$T_1^\prime$\end{tabular}}}}%
    \put(0.36213334,0.23450966){\color[rgb]{0,0,0}\makebox(0,0)[lt]{\lineheight{1.25}\smash{\begin{tabular}[t]{l}$T_2^\prime$\end{tabular}}}}%
    \put(0.13800928,0.03612231){\color[rgb]{0,0,0}\makebox(0,0)[lt]{\lineheight{1.25}\smash{\begin{tabular}[t]{l}$T_1$\end{tabular}}}}%
    \put(0.36106094,0.03612231){\color[rgb]{0,0,0}\makebox(0,0)[lt]{\lineheight{1.25}\smash{\begin{tabular}[t]{l}$T_2$\end{tabular}}}}%
    \put(0.89509821,0.29885153){\color[rgb]{0,0,0}\makebox(0,0)[lt]{\lineheight{1.25}\smash{\begin{tabular}[t]{l}$h_1(A)$\end{tabular}}}}%
  \end{picture}%
\endgroup%

  \caption{The result of gluing the white faces in $\overline{M_P}$ and $\overline{M_P}^\prime$ is shown on the left, with cylinders formed from truncated ideal vertices shown in grey (note that faces shown in white are black faces in $\overline{M_P}$ and $\overline{M_P}'$). From the second to third image we identify the black faces (shown as white). We see that if the cylinder came from a ideal vertex between two paired black triangles then the gluing corresponds to a crossing circle. }
  \label{Fig:GlueBlackFaces}
\end{figure}
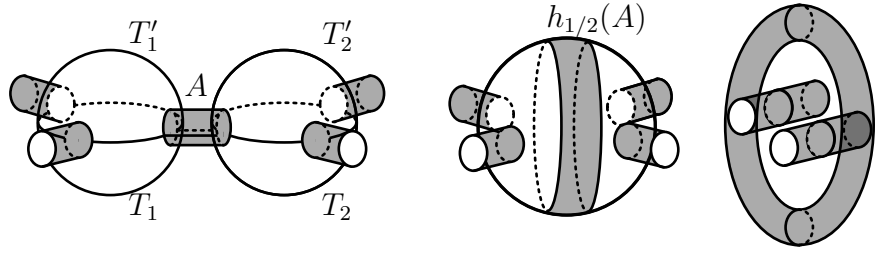
		
We can then perform an isotopy expanding $A$ so that $T_1 \cup T_1'$ and $T_2 \cup T_2'$ and $A$ lie on a sphere $S$ with $A$ forming a closed neighbourhood of a north-south great circle for $S$. We continue the isotopy, identifying $T_1\cup T_1'$ to $T_2\cup T_2'$ across a ball bounded by this sphere, as shown in \reffig{GlueBlackFaces}. This corresponds to identifying $T_1$ with $T_2$, and $T_1'$ with $T_2'$. Observe that the result after identification is a disk $D$ with with two open disks removed. The annulus $A$ has two boundary components identified to form a torus. This torus meets the black geodesic surface of $D$ on its outside boundary, corresponding to a longitude. The other two boundary components of $D$ correspond to two cylinders obtained by gluing vertices which do not pair black faces in the dimer. See \reffig{GlueBlackFaces}. Thus the ideal vertices that pair black triangles correspond to crossing circles.
		
Each of these steps is by isotopy in $\calD(\calM)$. We do this for each pair of truncated black triangles on $\bdy \overline{M_P}$. Hence the gluing of $\overline{M_P}$ and  $\overline{M_P}'$ gives a submanifold of $\calD(\calM)$. 
Finally, note that the gluing of $M_P$ now embeds as a submanifold of $\calD(\calM)$ because it is homeomorphic to the gluing of the truncated $\overline{M_P}$ without its boundary. 

We still need to show that $N$ is homeomorphic to a link complement in $\calD(\calM)$. We have seen that ideal vertices meeting paired black faces will correspond to crossing circles in $\calD(\calM)$. Now let $v_0 \in V$ be a vertex which does not pair two black faces. Let $R_{v_0}$ be the rectangle on $\bdy \overline{M_P}$ associated with $v_0$. Then $R_{v_0}$ meets two truncated black triangles $T_{-1}, T_1 \subset \bdy\overline{M_P}$. The triangle $T_1$ is paired to another truncated black triangle $T_2$ as specified by the dimer on $P$, across a vertex $v_1$. Similarly $T_{-1}$ is paired to another truncated black triangle $T_{-2}$, across a vertex $p_{-1}$. See \reffig{LinksInSurface}. After gluing $T_1$ and $T_2$, one of the black edges of $R_{v_0}$ will be glued to a black edge of another rectangle $R_{v_1}$ that intersects $T_2$, while the other black edge of $R_{v_0}$, will be glued to a black edge of a rectangle $R_{v_{-1}}$ that intersects $T_{-2}$.
		
\begin{figure}
  \centering
  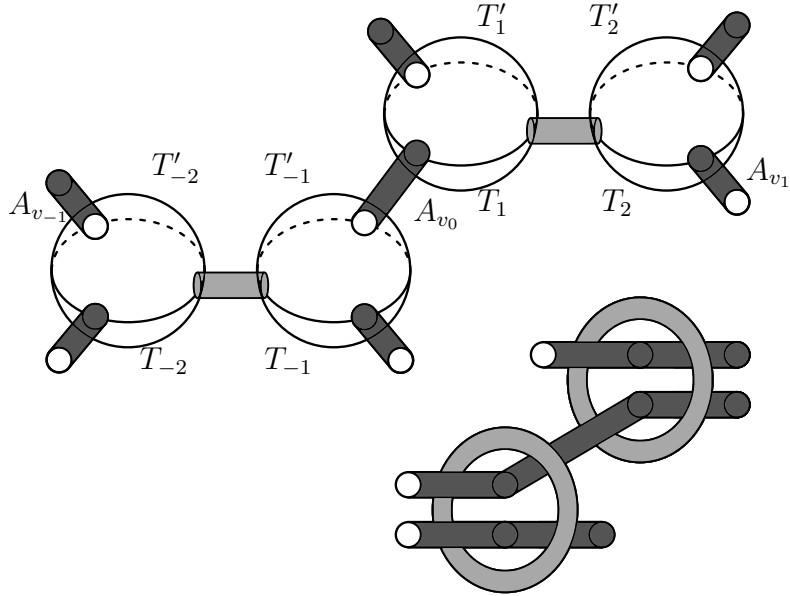
  \caption{First image shows the result after gluing truncated white faces (the truncated black triangles $T_i$ are shown as white). The light grey cylinders correspond to the cylinders associated with the paired vertices $v_{-1}$ and $v_1$. The dark grey cylinders do not pair black triangles together. The second image shows the result after gluing black triangles together.}
  \label{Fig:LinksInSurface}
\end{figure}

After gluing white faces, the pairs $R_{v_k}$ and $R_{v_k}'$, for $k \in \{-1,0,1\}$, are glued along their white edges and form cylinders, which we denote $A_k$ for $k \in \{-1,0,1\}$. After gluing the black faces, $A_{-1}$ will be glued to one end of $A_0$ while $A_1$ will be glued to the other end. Let $A$ be the the result of gluing these three cylinders together. The cylinder $A$ then passes through the two crossing circles associated with $v_{-1}$ and $v_{1}$. This is shown in the second image in \reffig{LinksInSurface}.
		
Every cylinder associated with a vertex $v \in V$ that does not pair black faces has its ends glued to other cylinders. It follows that the collection of all such cylinders forms a collection of tori. If $T$ is such a torus then $T$ has a Euclidean structure given by gluing a chain of rectangles $R_{v_0},R_{v_1}, \dots , R_{v_k}$ together; these are glued along their black sides. This chain is then glued to the corresponding chain $R_{v_0}',R_{v_1}', \dots , R_{v_k}'$ via their white sides. Note that the white sides of $R_{v_i}$ and $R_{v_i}'$, for $i \in \{0,1, \dots , k\}$ lie on a geodesic surface formed from gluing the white faces. In this sense each of these tori lies on the white surface formed from from the gluing of white faces, which is homeomorphic to $\bdy \calM$. Thus the glued manifold $N$ is homeomorphic to the complement of a fully augmented link on a surface without half twists. The ideal boundary components that correspond to vertices in $V$ pairing black faces are crossing circles, while the other vertices make up portions of the link components in the surface.
		
Finally we show that the resulting gluing has a complete hyperbolic structure. The fact that it has a hyperbolic structure follows from the fact that the gluing of faces is by isometry, and the faces meet at dihedral angle $\pi/2$, with four such angles identified under the gluing. Thus the sum of dihedral angles around any edge is $2\pi$; see for example~\cite[Theorem~4.7]{purcell2020hyperbolic}. 

To show that the structure is complete, we need to show that each of the ideal torus boundary components has an induced Euclidean structure; see for example~\cite[Theorem~4.10]{purcell2020hyperbolic}. We have seen that each torus boundary component is tiled by rectangles $R_v$ coming from ideal vertices of the scooped manifold. The cusp structure is induced by the gluing of the Euclidean rectangles. Since they are rectangles, with angles $\pi/2$, and matching side lengths, they do indeed give the cusp a Euclidean structure. 

Finally $N$ is finite volume since $M_P$ and $M_P\prime$ have finite volume, by \reflem{FiniteVolume}. Alternately since we have a complete hyperbolic 3-manifold with ideal boundary consisting of tori it must be finite volume; see for example~\cite[Theorem~5.24]{purcell2020hyperbolic}).
\end{proof}
	
One nice property of the links formed from this identification is that we can use the dimer on the nerve to draw the link directly from the circle packing.

\begin{corollary}
The link formed from the gluing of $M_{P}$ and $M_{P}^\prime$ can be drawn directly from the nerve of $P^*$ on $\Sigma$.
\end{corollary}

\begin{proof}
The nerve of $P^*$ is 3-valent with a coloured edge given by the dimer on $P$. Each coloured edge in $P^*$ corresponds to an ideal vertex shared by two paired black faces on $\partial M_P$. Such a vertex corresponds to a crossing circle. The two edges that are not coloured correspond to arcs in $\Sigma$. So for each coloured edge in $P^*$, draw a crossing circle, with arcs between crossing circles the non-coloured edges of $P^*$. Figure~\ref{Fig:LinkFromGraph} shows the local picture.
\end{proof}

\begin{figure}
  \centering
  \includegraphics{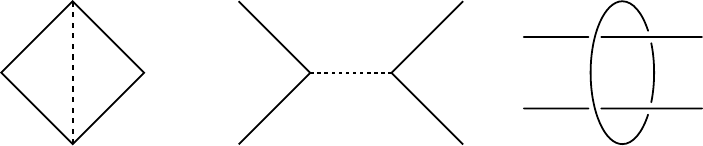}
  \caption{Left shows part of the nerve of $P$ in which the edge coloured by the dimer is shown dashed. Middle shows corresponding part of the nerve of $P^*$. Right, replace the dashed line with a crossing circle to draw the link.}
  \label{Fig:LinkFromGraph}
\end{figure}

\subsection{Adding Half-Twists}

\begin{lemma}\label{Lem:AddTwist}
  Let $C$ be a crossing circle of a fully augmented link $L$ embedded in a closed 3-manifold $M$ such that $M - L$ is hyperbolic. Then for the link $L'$ obtained by adding a half twist at $C$, the complement $M-L'$ is also hyperbolic.
\end{lemma}
	
\begin{proof}
This follows from Adams~\cite{adams1985thrice}. The crossing circle $C$ bounds a 3-punctured sphere, which is isotopic to a totally geodesic surface. Cut along this surface and reglue via the homeomorphism of the 3-punctured sphere that keeps the puncture associated with $C$ fixed and swaps the other two punctures. Since there is only one complete hyperbolic structure on a 3-punctured sphere, this is an isometry, hence gives a hyperbolic manifold with the desired properties.
\end{proof}

If we look back at the original gluing in \refthm{LinksInDouble}, adding a half twist at a crossing circle corresponds to changing the gluing of the black faces in $\bdy M_P$ and $\bdy M_P'$. Instead of gluing a black triangle to its pair on the same half, it will be glued to the pair in the opposite half.
	
\begin{lemma} \label{Lem:HalfTwistGluing}
Let $N$ be a manifold formed in the manner of Construction~\ref{Constr:FAL}, which are complements of fully augmented links without half-twists by \refthm{LinksInDouble}. Adding a half twist at a crossing circle corresponds to gluing a black triangle $T_1$ of $M_P$ with the triangle $T_2'$ on $M_P'$, paired to the reflection $T_1'$ of $T_1$ by the dimer. 
\end{lemma}
	
\begin{proof}
  A half-twist is added by rotating the half $T_1\cup T_1'$ of \reffig{GlueBlackFaces}, middle, by $180^\circ$ before gluing. See \reffig{HalfTwistGluing}. This glues $T_1$ with $T_2'$, and $T_1'$ with $T_2$, via an orientation reversing isometry. 
\end{proof}
		
\begin{figure}
  \centering
\begingroup%
  \makeatletter%
  \providecommand\color[2][]{%
    \errmessage{(Inkscape) Color is used for the text in Inkscape, but the package 'color.sty' is not loaded}%
    \renewcommand\color[2][]{}%
  }%
  \providecommand\transparent[1]{%
    \errmessage{(Inkscape) Transparency is used (non-zero) for the text in Inkscape, but the package 'transparent.sty' is not loaded}%
    \renewcommand\transparent[1]{}%
  }%
  \providecommand\rotatebox[2]{#2}%
  \newcommand*\fsize{\dimexpr\f@size pt\relax}%
  \newcommand*\lineheight[1]{\fontsize{\fsize}{#1\fsize}\selectfont}%
  \ifx\svgwidth\undefined%
    \setlength{\unitlength}{231.4385722bp}%
    \ifx\svgscale\undefined%
      \relax%
    \else%
      \setlength{\unitlength}{\unitlength * \real{\svgscale}}%
    \fi%
  \else%
    \setlength{\unitlength}{\svgwidth}%
  \fi%
  \global\let\svgwidth\undefined%
  \global\let\svgscale\undefined%
  \makeatother%
  \begin{picture}(1,0.17109404)%
    \lineheight{1}%
    \setlength\tabcolsep{0pt}%
    \put(0,0){\includegraphics[width=\unitlength,page=1]{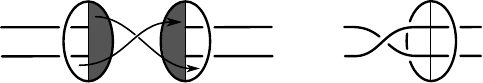}}%
    \put(0.22901383,0.14636847){\color[rgb]{0,0,0}\makebox(0,0)[lt]{\lineheight{1.25}\smash{\begin{tabular}[t]{l}$T_1$\end{tabular}}}}%
    \put(0.07232791,0.01321053){\color[rgb]{0,0,0}\makebox(0,0)[lt]{\lineheight{1.25}\smash{\begin{tabular}[t]{l}$T_1^\prime$\end{tabular}}}}%
    \put(0.43763384,0.01456342){\color[rgb]{0,0,0}\makebox(0,0)[lt]{\lineheight{1.25}\smash{\begin{tabular}[t]{l}$T_2^\prime$\end{tabular}}}}%
    \put(0.28328272,0.14596339){\color[rgb]{0,0,0}\makebox(0,0)[lt]{\lineheight{1.25}\smash{\begin{tabular}[t]{l}$T_2$\end{tabular}}}}%
  \end{picture}%
\endgroup%

  \caption{Shows how gluing black triangles in $\bdy M_P$ to the paired triangle in $\bdy M_P^\prime$ corresponds to adding a half twist.}
  \label{Fig:HalfTwistGluing}
\end{figure}
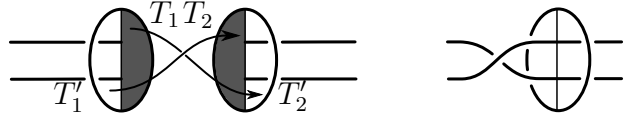
	
\begin{lemma}\label{Lem:SingleComponent}
Let $M=\HH^3/\Gamma$ be a convex cocompact hyperbolic 3-manifold. Let $N$ be the complement of a fully augmented link in $\calD(\calM)$ constructed in Construction~\ref{Constr:FAL}. Then we may form a new hyperbolic 3-manifold $N'$ such that $N'$ is the complement of a fully augmented link $L'$ on $\bdy\calM \subset \calD(\calM)$, where $L'$ has only one component that is not a crossing circle on each component of $\partial \mathcal{M}$, and $L'$ is formed from $L$ by adding half twists at some of the crossing circles of $L$.
\end{lemma}
	
\begin{proof}
Let $K_1, \dots , K_n$ be the link components of $L$ that are not crossing circles. If $n \geq 2$, then since the diagram graph of $L$ is connected, there must be some crossing circle $C$ such that there are two distinct components $K_j$ and $K_k$ passing through $C$. Let $L_C$ denote the link formed by adding a half twist at $C$ to $L$. Adding the half twist at $C$ concatenates $K_j$ and $K_k$, reducing the number of components by one. Repeat until there is only one component that is not a crossing circle on each component of $\partial \mathcal{M}$.
\end{proof}

\subsection{Showing Geometric Convergence}
	
Now we show how we can use the construction of the previous section to construct sequences of link complements which converge geometrically to $M$.
	
\begin{lemma}\label{Lem:LinksConvergeDouble}
Let $M=\mathbb{H}^3/\Gamma$ be a convex-cocompact
hyperbolic 3-manifold homeomorphic to the interior of a compact 3-manifold $\overline M$ and let $\epsilon>0$ and $R>0$.

Then there exists a finite volume hyperbolic 3-manifold with framed basepoint $(M_{\epsilon,R},p_{\epsilon,R})$ that is a link complement in $\calD(\overline M)$
such that $(M_{\epsilon,R},p_{\epsilon,R})$ is $(\epsilon, R)$-close to $(M,p)$, where $p$ is the framed basepoint on $M=\HH^3/\Gamma$ induced by $O$ in $\HH^3$.
\end{lemma}

\begin{proof}
By \reflem{ScoopedContainingBall}, we can find an $e^\delta$-quasiconformal homeomorphism $\phi$ fixing $0,1,\infty$ conjugating $\Gamma$ to $\Gamma_\delta$
such that the associated convex-cocompact manifold $N_{\delta}=\HH^3/\Gamma_{\delta}$ admits a circle packing $P_{\delta}$ on its conformal boundary, and the metric ball $B(0,R)/\Gamma_\delta$ is completely contained in the corresponding scooped manifold $(N_{\delta})_{P_\delta}$.
Further, we may take $N_{\delta}$, $P_{\delta}$ as above so that the nerve of $P_{\delta}$ admits a dimer. By Corollary \ref{Cor:close}, $N_{\delta}$ is $(\varepsilon,R)$-close to $M$ for $\delta$ sufficiently small, if both $M=\HH^3/\Gamma$ and $N_{\delta}=\HH^3/\Gamma_{\delta}$ are endowed with the framed basepoint $p,p_\delta$ induced from $O$ in $\HH^3$.

Let $M_{\epsilon,R}$ be a link complement in $\calD(\overline{M})$ formed from gluing two copies of $(N_{\delta})_{P_\delta}$ in the manner specified in \refthm{LinksInDouble} for $\delta=\delta(\epsilon,R)$ small as above. Since $(N_{\delta})_{P_\delta}$ isometrically embeds in $M_{\epsilon,R}$, we have (denoting the image of $p_\delta$ by $p_{\epsilon,R}$) that $(M_{\epsilon, R},p_{\epsilon,R})$ is $(\epsilon,R)$-close to $(M,p)$. 
\end{proof}

As an immediate consequence we have:
\begin{corollary}\label{Cor:LinksConvergeDouble}
  The links of \reflem{LinksConvergeDouble} converge geometrically to $M$.\qed
\end{corollary}

We now turn the link complements of \refcor{LinksConvergeDouble} into knot complements. 

\begin{theorem}\label{Thm:ConvergenceDouble}
Let $M$ be a convex cocompact hyperbolic 3-manifold that is the interior of a compact 3-manifold $\overline{M}$.  
Then there exists a sequence of finite volume hyperbolic 3-manifolds $M_n$ that are link complements in $\calD (\overline M)$, with one link component per boundary component of $\overline{M}$, such that $ M_n $ converges geometrically to $M$.

In particular, if $\overline{M}$ has a single boundary component, then $M$ is the geometric limit of a sequence of knot complements. 
\end{theorem}
	
\begin{proof}
By taking $(\epsilon,R)=(1/n,n)$ in \reflem{LinksConvergeDouble}, we find a sequence of fully augmented links on a surface in $\calD(\overline{M})$ which contain $(n+1)/n$-bilipschitz images $B(p,n)\subset M$. By \reflem{SingleComponent}, by adding half twists at some of the crossing circles we obtain a fully augmented link on the surface $\bdy\overline{M}\subset \calD(\overline{M})$ that has a single component that is not a crossing circle on each component of $\bdy\overline{M}$. Lemma~\ref{Lem:HalfTwistGluing} shows that adding a half twist corresponds to changing the gluing of black faces, which does not affect the embedding $B(p,n)$ of \reflem{LinksConvergeDouble}. Thus we obtain a sequence $L_n$ of complements of fully augmented links in $\calD (\overline{M})$ converging geometrically to $(M,p)$, for suitable framed basepoints, such that for each component of $\bdy\overline{M}$ embedded in $\calD(\overline{M})$, only one link component is not a crossing circle.

Let $s\in \ZZ$ be a positive integer. Observe that $1/s$ Dehn filling on a crossing circle $C$ of $L_n$ inserts $2s$ crossings into the twist region encircled by $C$ and removes the link component $C$. We do this for all crossing circles. Let $i_n$ be the number of crossing circles in $L_n$, and let $ s_k^1, \dots , s_k^{i_n} $ denote sequences of positive integers approaching infinity as $k\rightarrow \infty$. Thurston's hyperbolic Dehn surgery theorem tells us that for fixed $n$ the sequence of manifolds $ M_n(1/s_k^1, \dots , 1/s_k^{i_n}) $ converges geometrically to $M_n$~\cite{thurston1979geometry}. Taking a diagonal sequence, we obtain a sequence of knot complements in $\calD (\overline{M})$ converging geometrically to $M$.
\end{proof}

\subsection{Effective Dehn filling}
We promised in the introduction a constructive method to build knot complements converging to $M$. Theorem~\ref{Thm:ConvergenceDouble} uses Thurston's hyperbolic Dehn surgery theorem to imply that such knots must exist, however that theorem is not constructive. In this section, we explain how the proof can be modified to use cone deformation techniques to explicitly construct knots with the desired properties.

To do so, we need to know more about the cusp shapes and normalised lengths of Dehn filling slopes on the link complements $M_{\epsilon,R}$ of \reflem{LinksConvergeDouble}.

\begin{lemma}\label{Lem:CuspTiling}
  In the hyperbolic structure on the fully augmented link complement $M_{\epsilon,R}$ of \reflem{LinksConvergeDouble}, each cusp corresponding to a crossing circle is tiled by two identical Euclidean rectangles. Each rectangle has a pair of opposite sides coming from the intersection of a horospherical cusp torus with black sides, and a pair coming from an intersection with white sides. The slope $1/n$ on this cusp is isotopic to a curve as follows:
  \begin{itemize}
  \item If the crossing circle does not meet a half-twist, the slope is given by one step along a white side, plus or minus $2n$ steps along black sides. 
  \item If the crossing circle meets a half-twist, then the meridian is sheared. Thus the slope is given by one step along a white side, plus or minus $(2n+1)$ steps along black sides.
  \end{itemize}
  In either case, if $c$ is the number of crossings added to this twist region of the diagram after Dehn filling, then the slope is given by one step along a white side plus or minus $c$ steps along black sides.
\end{lemma}

\begin{proof}
The proof is completely analogous to a similar result for crossing circle cusps in the classical setting of fully augmented links in the 3-sphere; see \cite[Proposition~3.2]{Purcell:AugLinksSurvey} or \cite[Lemma~2.6, Theorem~2.7]{FuterPurcell:Exceptional}. We walk through it in this setting.

By \reflem{RectangleVertices}, each crossing circle is tiled by rectangles, each with two opposite black sides, coming from intersections of black triangles with a horospherical torus about the cusp, and two opposite white sides, coming from intersections of white faces with a horospherical torus. Tracing through the gluing construction of \ref{Constr:FAL}, with reference to \reffig{GlueBlackFaces}, the crossing circle cusps are built by first gluing one rectangle from the original scooped manifold $M_P$ to an identical copy from $M_P'$, via a reflection in a white side. When there is no half-twist, the black sides of each of these rectangles are then glued together. A longitude runs over the two black sides, meeting two white sides along the way. A meridian runs over exactly one white side, meeting exactly one black side transversely along the way. 

When a half-twist is added, the longitude still runs over two black sides, but a meridian is obtained by taking a step along a white side plus or minus a step along a black side, depending on the direction of twist. We may assume that the direction of twist matches the sign of $n$, otherwise apply a homeomorphism giving a half-twist in the opposite direction, and reduce $|n|$ by two. This introduces shearing to the meridian. 

The slope $1/n$ runs over one meridian and $n$ longitudes. In the case of no half-twists, this is one step along a white side, plus $2n$ steps along black sides. This adds $|2n| =c$ crossings to the twist region.

When there is a half-twist, the slope $1/n$ still runs over one meridian plus $n$ longitudes, but now this is given by one step along a white side plus or minus one step along a black side (with sign matching sign of $n$), plus $2n$ additional steps along black sides. Again there are $c=|2n+1|$ steps along black sides. 
\end{proof}

The \emph{normalised length} of a slope $s$ on a cusp torus $T$ is the length of a geodesic representative of the slope in the Euclidean metric on $T$, divided by the area of the torus:
\[ L(s) = \len(s) / \sqrt{\area(T)}. \]
Observe that the normalised length is independent of scale, thus it is an invariant of the cusp rather than the choice of horospherical neighbourhood of the cusp.

The following result, for fully augmented links in $\calD(M)$, is analagous to a calculation for fully augmented links in $S^3$ found in \cite{Purcell:Volumes2007}. 

\begin{lemma}\label{Lem:NormLength}
  Let $c$ be the number of crossings added by Dehn filling at a crossing circle. Then the corresponding slope of the Dehn filling has normalised length at least $\sqrt{c}$.
\end{lemma}

\begin{proof}
From \reflem{CuspTiling}, we know that the two rectangles in the cusp tiling of the crossing circle are identical, hence each white side has length $w$ and each black side length $b$. The area of the cusp, with or without half-twists, is given by $2bw$. Thus by \reflem{CuspTiling}, the normalised length of the slope $1/n$ is given by
\[ L = \frac{\sqrt{w^2 + c^2 b^2}}{\sqrt{2bw}} = \sqrt{\frac{w}{2b}+\frac{c^2b}{2w}}. \]
This is minimised when $w/2b$ equals $c/2$, and the minimum value is $\sqrt{c}$. 
\end{proof}

\begin{lemma}\label{Lem:UnivCrossing}
Given $M$, $\epsilon>0, R>0$, and $M_{\epsilon,R}$ as in \reflem{LinksConvergeDouble}, let $\delta>0$ be such that
$B(p,R)$ lies in the $\delta/(1+\epsilon)$-thick part of $M$.
Let $n$ denote the number of crossing circles of the fully augmented link in $M_{\epsilon,R}$. If after Dehn filling the crossing circles, the number of crossings added to each twist region is at least
\[ n\cdot \max \left\{ \frac{107.6}{\delta^2}+14.41,
\frac{45.20}{\delta^{5/2}\log(1+\epsilon)}+ 14.41 \right\}, \]
then the inclusion map taking $B(p_{\epsilon,R},R)$ in $M_{\epsilon,R}$ into the complement of the resulting knot in $\calD(M)$ is $(1+\epsilon)$-bilipschitz. It follows that the knot complement contains a set that is $(1+\epsilon)^2$-bilipschitz to $B(p,R)$ in the original $M$.
\end{lemma}

\begin{proof}
By \reflem{LinksConvergeDouble}, if $B(p,R)$ lies in the $\delta/(1+\epsilon)$ thick part of $M$, then $B(p_{\epsilon,R},R)$ lies in the $\delta$ thick part of $M_{\epsilon,R}$ and is $(1+\epsilon)$ bilipschitz to $B(p,R)$. 
  
Let $L^2$ be given by 
\[ \frac{1}{L^2} = \sum_{i=1}^n \frac{1}{L_i^2}, \]
where $L_i$ is the normalised length of the Dehn filling slope on the $i$-th crossing circle cusp. In \cite[Corollary~8.16]{FPS:EffectiveBilipschitz}, it is shown that if $L^2$ is at least the maximum given above, then the inclusion map on any submanifold of the $\delta$-thick part is $(1+\epsilon)$-bilipschitz.

Let $C$ be the minimal number of crossings added to any twist region. By \reflem{NormLength}, $1/L_i^2 \leq 1/C$, so $1/L^2 \leq n/C$, or $L^2 \geq C/n$. Thus if $C/n$ is at least the maximum in the formula above, we may apply the corollary from \cite{FPS:EffectiveBilipschitz} to $B(p_{\epsilon,R},R)$.
\end{proof}

\section{Reducing geometrically finite to convex cocompact}\label{Sec:GFtoCC}

The previous sections constructed link complements that converge to convex cocompact hyperbolic structures. In the case of a single topological end, the limiting manifolds are all knot complements. The construction can be extended almost immediately to geometrically finite manifolds of infinite volume. However, now in the case that the manifold has a single topological end, if that end contains a rank-1 cusp, the immediate extension produces link complements rather than knot complements.
Indeed, in the presence of rank one and rank two cusps our construction above leads to several cusp boundary components and thus to a complementary link with multiple components. 
Instead we will show that a geometrically finite manifold $M$ can be approximated geometrically by convex cocompact manifolds. Combining this with the previous results, it follows that $M$ can also be approximated geometrically by knot complements if it is of infinite volume with a single topological end.

For rank two cusps, a version of Thurston's hyperbolic Dehn surgery theorem for geometrically finite hyperbolic manifolds shows that a geometrically finite manifold is the geometric limit of geometrically finite manifolds without rank two cusps; see, for example, work of Brock and Bromberg~\cite{BrockBromberg}. However in our setting, i.e.\ a 3-manifold with one end, rank one cusps are more problematic. Here we show that for any geometrically finite hyperbolic manifold $M$, there is sequence of geometrically finite hyperbolic manifolds $M_j$ without rank one cusps converging to $M$. Moreover the sequence can be chosen such that the maps establishing this convergence are global diffeomorphisms.
In particular $M_j$ is diffeomorphic to $M$ for each $j$.

Results such as this go back to work of J{\o}rgensen, and is presumably implicit in the construction of Earle--Marden geometric coordinates (c.f.\ \cite{Mardencoordinates} and the appendix of \cite{HubbardKoch}); compare also Marden~\cite[exercises 4-24 and 5-3]{marden2007outer}. We include the result and a proof for completeness.

\begin{theorem}\label{Thm:CCtoGF}
Let $M$ be a geometrically finite hyperbolic manifold. Then there exists a sequence of geometrically finite hyperbolic manifolds $M_j$ without any rank one cusps and diffeomorphisms $M\rightarrow M_j$ establishing that the $M_j$ converge geometrically to $M$. The $M_j$ are explicitly constructed starting from $M$ and there are effective bounds for the convergence.
\end{theorem}

To prove \refthm{CCtoGF}, we first need to set up some notation. Fix a framed basepoint on $p$ on $M$. Then $(M,p)$ corresponds to a Kleinian group $\Gamma$ such that 
$(M,p)=(\mathbb{H}^3/\Gamma,O)$. We will first construct Kleinian groups $\Gamma_{r(j)n(j)}$ corresponding to suitable hyperbolic 3-manifolds with framed basepoints $(M_j,p_j)$ that converge to $\Gamma$ in the Chabauty topology (and thus $(M_j,p_j)$ converges geometrically to $(M,p)$). When viewed as perturbations of $\Gamma$, the Kleinian groups $\Gamma_{r(j)n(j)}$ also converge algebraically to $\Gamma$ and the desired convergence properties will follow. 
  
Consider a fixed rank one cusp of $M$, generated by $\eta_1$. Up to conjugation, we may assume $\eta_1$ corresponds to $z\mapsto z+1$. For $r_1>0$, let $\gamma_{r_1}$ correspond to $z\mapsto z+r_1 \sqrt{-1}$. Add $\gamma_1:=\gamma_{r_1}$ to $\Gamma$ as a generator to obtain $\Gamma_{r_1}$, with presentation $\langle G,\gamma_1 \mid R, [\gamma_1,\eta_1]=1\rangle$, where $\langle G \mid R \rangle$ is a presentation of $\Gamma$.

\begin{lemma}\label{Lem:KleinMaskit}
For $r_1$ sufficiently large, $\Gamma_{r_1}$ is a discrete group and an HNN extension of $\Gamma$.
\end{lemma}

\begin{proof}
This will be a consequence of the second Klein-Maskit combination theorem; we use the version as stated in Abikoff and Maskit~\cite{AbikoffMaskit}, for a proof see Maskit \cite[VII~E.5]{MaskitBook}.

Let $H$ be a subgroup of $\Gamma$. Recall that a subset $B\subset \CC\cup \{\infty\}$ is precisely invariant under $H$ in $\Gamma$ if (1) for all $h\in H$, $h(B)=B$ and (2) for all $\gamma \in \Gamma\setminus H$, $\gamma(B)\cap B = \varnothing$.
In our setting, consider the round discs 
$D_\pm:=D_\pm(r_1)=\{z\in \CC \mid \pm\mbox{Im}(z)>r_1/2\}$  in $\CC\cup \{\infty\}$. 
We claim that for $r_1$ sufficiently large, the $\Gamma$-orbits of $D_+$ and $D_-$ are disjoint and that $D_\pm$ are both precisely invariant under the subgroup $H=\langle \eta_1 \rangle$ of $\Gamma$.

This follows, for example, from work of Bowditch~\cite{Bowditch}, specifically his result that geometrically finite is equivalent to his definition~GF1, which we now recall. By Bowditch's definition GF1, the fundamental domain of a geometrically finite hyperbolic manifold is realised as the union of a compact set and a finite number of disjoint \emph{standard} cusp regions (c.f.\ \cite[Proposition 4.4]{Bowditch} for a proof that geometrically finite hyperbolic manifolds admit standard cusp regions). A standard cusp for $\eta_1$ is modelled as follows.
Consider the universal cover $\HH^3$ of $M$, in the upper half-space model, with boundary $\CC \cup \{\infty\}$. The parabolic $\eta_1$, taking $z$ to $z+1$, acts as translation on horospheres about infinity, taking vertical planes in $\HH^3$ with boundary of the form $\{x\in \CC \mid \mbox{Re}(x)=R\}$, for fixed $R\in \RR$, to vertical planes in $\HH^3$ with boundary $\{x\in \CC \mid \mbox{Re}(x)=R+1\}$. 
There is an $\eta_1$-invariant subspace $P\subset \CC$ with $P/\langle \eta_1 \rangle$ compact; in the 3-dimensional rank-1 case at hand, $P=P(r)$ can be chosen to be an infinite strip bounded by two lines $L(\pm r/2)=\{x\in \CC \mid \mbox{Im}(x) = \pm r/2\}$. See \cite[Figure~3a]{Bowditch}.
Bowditch's definition of a standard cusp implies that for some height $h>0$, the region 
\[C=C(P(r),h)=\{x\in\HH^3 \mid d_{\mbox{euc}}(x,P(r))\geq h\}\]
 must satisfy $\gamma(C)\cap C = \varnothing$ for all $\gamma \in \Gamma\setminus H$. For $r_1$ large, $D_\pm \subset C$, and therefore $\gamma(D_+\cup D_-)\cap (D_+\cup D_-)=\varnothing$ for $\gamma\in \Gamma\setminus H$. Combining this with the fact that $H$ preserves both $D_\pm$ separately, it follows that the $\Gamma$-orbits of $D_\pm$ are disjoint and that both $D_\pm$ are precisely invariant under $H$ in $\Gamma$.

Now consider $f=\gamma_{r_1}$ defined as above. Note that since $\gamma_{r_1}$ and $\eta_1$ commute, $fHf^{-1} =H < \Gamma$. 
The observations above on Bowditch's definition~GF1 imply the following three conditions required for the second Klein-Maskit combination theorem:
\begin{enumerate}
\item $D_+$ is precisely invariant for $H$ in $\Gamma$; 
\item $\CC-\gamma_{r_1}(\overline{D}_+)=D_-$ is precisely invariant for $fHf^{-1} = H$ in $\Gamma$
\item $\gamma(D_+) \cap D_- = \varnothing$ for all $\gamma\in\Gamma$.
\end{enumerate}
Then by the second Klein-Maskit combination theorem, $\Gamma_{r_1}$ is a discrete group and an HNN extension of $\Gamma$. 
\end{proof}

\begin{proof}[Proof of \refthm{CCtoGF}]
Apply \reflem{KleinMaskit} iteratively to all rank one cusps of $\Gamma$; we obtain a Kleinian group $\Gamma_r$, $r=(r_1,\ldots , r_k)$, for $r_{i+1}\gg r_i$, $i=1,\ldots,k-1$. It has $k$ rank two cusps corresponding to the $k$ rank one cusps of $M$, and additionally any rank two cusps inherited from $\Gamma$, but no rank one cusps. It has a presentation of the form
\[ \Gamma_r=\langle G, \gamma_1,\ldots , \gamma_k\,\vert\, R, [\gamma_i,\eta_i]=1,  \forall i=1,\ldots, k\rangle. \]
As $r_1=\min_i r_i$ tends to infinity, these groups converge geometrically to $\Gamma$.

Now perform $(1,n)$-Dehn surgery on the $k$ new rank two cusps of $\Gamma_r$, where the meridian of the  $i$-th cusp (filled for $n=0$) corresponds to the new generator $\gamma_{r_i}$. For $n$ sufficiently large, this yields Kleinian groups $\Gamma_{rn}$ with presentations
\[
\Gamma_{rn}=\langle G, \gamma_1,\ldots , \gamma_k\,\vert\, R, [\gamma_i,\eta_i]=1,\gamma_i\eta_i^n=1,\,  \forall i=1,\ldots, k\rangle.
\]
The groups $\Gamma_{rn}$ are canonically isomorphic to $\Gamma$: There is a natural isomorphism $m_{rn}:\Gamma\rightarrow \Gamma_{rn}$ whose inverse sends $\gamma_i $ to $\eta_i^{-n}$ for all $i=1,\ldots, k.$

Thus the $\Gamma_{rn}$ are images of faithful, geometrically finite representations of $\Gamma$.
Moreover, since the construction of $\Gamma_{rn}$ is via Dehn surgery, for $n$ large, $m_{rn}(\theta)$ for $\theta\in\Gamma$ is parabolic if and only if $\theta$ is part of a rank two cusp of $\Gamma$. In particular, the elements $m_{rn}(\eta_i)$ are hyperbolic and $\Gamma_{rn}$ has no rank one cusps.

These representations converge algebraically to $\Gamma$ as $n\rightarrow \infty$, since Dehn surgery is a perturbation of the identity in terms of representations of the group $\Gamma_r$, thus in particular in terms of the subgroups $\Gamma \subset \Gamma_r$. A suitable formulation of Dehn surgery, due to Comar, can be found in \cite[Theorem~10.1]{AndersonCanaryMcCTopologyDeformation}.
Moreover the Kleinian groups $\Gamma_{rn}$ converge geometrically (i.e.\ in the Chabauty topology) to $\Gamma_r$ as $n\rightarrow \infty$~\cite{thurston1979geometry}.
Thus for each value of $r_i$, we may choose a sequence $r_{i}(j)_{ j\in \mathbb{N}}$ tending to infinity, and consider $r(j)=(r_1(j),\ldots,r_k(j))$ as above. Choosing $n(j)$ sufficiently large, we find that the
diagonal sequence of Kleinian groups $\Gamma_{r(j)n(j)}$, uniformizing the geometrically finite hyperbolic manifolds $M_j$ without rank one cusps, converges both geometrically and algebraically to $\Gamma$, uniformizing $M$.

This implies that the limit $M$ is diffeomorphic to $M_j$ for $j$ sufficiently large, as follows (compare \cite[Lemma~3.6]{AndersonCanaryMcCTopologyDeformation}).
Indeed, the compact core of $M$ embeds via its interpretation as geometric limit back into $M_j$ for $j$ large. This induces a map on fundamental groups $\Gamma\rightarrow \Gamma_{r(j)n(j)}$, which necessarily coincides with the isomorphism $\Gamma\rightarrow \Gamma_{r(j)n(j)}$ establishing that $\Gamma$ is the algebraic limit of $\Gamma_{r(j)n(j)}$.
Thus the compact core of $M$ embeds as a compact core into $M_j$ for $j$ large. By the uniqueness of compact cores and since a diffeomorphism of compact cores can be extended to a diffeomorphism of the ambient hyperbolic manifolds, the claimed result follows.

Finally we remark on the constructive nature of the proof. Observe that the process above is obtained by first, choosing a sufficiently large $r_i$ at each rank one cusp
to build manifolds with rank two cusps. Then perform high Dehn filling. The choice of $r_1$ will depend heavily on $M$, but given a fundamental domain for $M$, these can be determined effectively. By our choice of the $\gamma_{r_i}$, the new rank two cusps of the manifold $\HH^3/\Gamma_r$ are rectangular. Thus the normalised length of the slopes $1/n$ have length at least $\sqrt{n}$. Again applying cone deformation techniques, we may choose effective $n$ sufficiently large to obtain constants required in the definition of geometric convergence, as in the proof of \reflem{UnivCrossing}.
\end{proof}

\begin{corollary}\label{Cor:MainGF}
  Let $M$ be a geometrically finite hyperbolic 3-manifold of infinite volume that is homeomorphic to the interior of a compact manifold $\overline{M}$ with a single boundary component. Then one can construct an explicit sequence of finite volume hyperbolic manifolds that are knot complements in $\mathcal{D}(\overline{M})$ such that $M_n$ converges geometrically to $M$. \qed
\end{corollary}

\bibliographystyle{amsplain}
\bibliography{biblio}

\end{document}